\documentclass[11pt,letterpaper]{amsart}

\title{Decomposition of Borel graphs and cohomology}
\author{Hiroki Ishikura}
\date{}

\usepackage{amsmath,amssymb,amsthm,amscd,enumitem,url,hyperref,mleftright}
\usepackage[margin=30truemm]{geometry}

\numberwithin{equation}{section}

\makeatletter
\newsavebox{\@brx}
\newcommand{\llangle}[1][]{\savebox{\@brx}{\(\m@th{#1\langle}\)}%
  \mathopen{\copy\@brx\kern-0.5\wd\@brx\usebox{\@brx}}}
\newcommand{\rrangle}[1][]{\savebox{\@brx}{\(\m@th{#1\rangle}\)}%
  \mathclose{\copy\@brx\kern-0.5\wd\@brx\usebox{\@brx}}}
\makeatother

\theoremstyle{plain}
\newtheorem{thm}{Theorem}[section]
\newtheorem{prop}[thm]{Proposition}

\newtheorem{lem}[thm]{Lemma}
\newtheorem{cor}[thm]{Corollary}
\newtheorem{thmA}{Theorem}

\theoremstyle{definition}
\newtheorem{dfn}[thm]{Definition}
\newtheorem{nta}[thm]{Notation}

\theoremstyle{remark}
\newtheorem{rmk}[thm]{Remark}
\newtheorem{exm}[thm]{Example}

\begin{document}

\begin{abstract}
We give a cohomological criterion for certain decomposition of Borel graphs, which is an analog of Dunwoody's work on accessibility of groups. As an application, we prove that a Borel graph $(X,G)$ with uniformly bounded degrees of cohomological dimension one is Lipschitz equivalent to a Borel acyclic graph on $X$. This gives a new proof of a result of Chen-Poulin-Tao-Tserunyan on Borel graphs with components quasi-isometric to trees.
\end{abstract}

\maketitle

\section{Introduction} \label{sintr}

\subsection{Decomposition of Borel graphs}

\subsubsection*{Accessibility of groups}

The Stallings theorem for ends of groups \cite{Sta,Ber} states that a finitely generated group with more than one ends decomposes into either an amalgamated free product or an HNN extension over a finite subgroup. Dunwoody's accessibility is the notion that requires this decomposition process to finish in finite steps. 

\begin{dfn}
A group is \textit{accessible} if it admits a co-compact action on a tree such that all vertex stabilizers are finitely generated groups with at most one end, and all edge stabilizers are finite groups.
\end{dfn}

There is a cohomological characterization of accessibility. All rings in this paper are assumed to be unital.

\begin{thm}[{\cite[Theorem 5.5]{Dun1}}]\label{intr4}
Let $R$ be a non-zero commutative ring. Then a finitely generated group $\Gamma$ is accessible if and only if the cohomology group $\textup{H}^1(\Gamma,R\Gamma)$ is finitely generated as a right $R\Gamma$-module.
\end{thm}

\subsubsection*{Borel graphs.} 
Let $X$ be a standard Borel space. A \textit{Borel equivalence relation} on $X$ is an equivalence relation on $X$ which is a Borel subset of $X\times X$. It is called a \textit{countable Borel equivalence relation} if each equivalence class is at most countable. They have been studied in the context of ergodic theory and descriptive set theory. A Borel graph on $X$ is a simplicial graph on $X$ whose edge set is a Borel subset of $X\times X$. If it is locally countable, then its connected relation defines a countable Borel equivalence relation on $X$. This is a generalization of the orbit equivalence relation associated with a Borel action of a countable group on $X$. One of the main theme of this area is to see how combinatorial or geometric structures of Borel graphs are reflected in the equivalence relations that they generate. Recently, there have been many attempts to apply ideas of geometric group theory to the study of Borel graphs (e.g., \cite{CJM}, \cite{CGMT}).

Tserunyan establishes an analog of the Stallings theorem for Borel graphs \cite[Theorem 4.1]{Tse} (see also \cite[Th\'{e}or\`{e}me D]{Ghy} and \cite[Theorem C]{Pau} for other analogs). It is proved that if a locally finite Borel graph $G$ has components with more than one ends, then the Borel equivalence relation generated by $G$ is a free product of two subequivalence relations, one of which is \textit{treeable}, i.e., generated by a Borel acyclic graph.

By developing this idea, we obtain an analog of Theorem \ref{intr4} for Borel graphs.

\begin{thmA}[Theorem \ref{mos3}] \label{intr1}
Let $R$ be a non-zero commutative ring and $(X,G)$ a Borel graph with uniformly bounded degrees. Suppose that $\textup{H}^1(G,R_G)$ is finitely generated as a right $R_G$-module. Then there exists an injective Borel quasi-isometry $\gamma:(X,G)\to (Y,G')$, where $(Y,G')$ is a Borel graph with uniformly bounded degrees such that:
\begin{enumerate}
    \item $G'=T\ast H$ with $T$ and $H$ Borel subgraphs of $G'$.
    \item $T$ is acyclic.
    \item $H$ is uniformly at most one-ended.
\end{enumerate}
\end{thmA}

Here $R_G$ is an algebra associated to $G$, which is an analog of the group ring for a group. As explained below, it is defined only by the metric strucure of $(X,G)$ as an abstarct graph, that is, its measurable structure is forgotten. Hence the cohomological assumption of Theorem \ref{intr1} is only involved with the metric structure.\par
The consequence of the theorem states that $G$ admits a nice decomposition into two Borel graphs $T$ and $H$, which is an analog of Bass-Serre decomposition of a group. Condition (i) and (ii) means that a vertex of the Bass-Serre tree correspond to an $E_H$-class and an edge of the Bass-Serre tree correspond to an edge of $T$. Finally, condition (iii) means that the decomposition is optimal, i.e., $H$ cannot be decomposed essentially anymore.\par
There are two major contributions of us. First, we not only give a free product decomposition $E_1\ast E_2$, but also give two Borel graphs $T$ and $H$ generating $E_1$ and $E_2$ respectively such that the union of $T$ and $H$ is ``close" to $G$. To be precise, the construction of $T$ basically follows \cite[Theorem 4.1]{Tse}, but as far as we understand, the construction $H$ appears first in this paper (see Proposition \ref{lot8}). Second, we establish the notion of ``optimal decomposition" and characterize when it is possible. This is a crucial part of the paper, where cohomology theory plays a role. Further, the optimality of the decomposition leads to an application in section \ref{sintr2}.

\begin{rmk}[Proposition \ref{bk14}]
The converse of Theorem \ref{intr1} also holds. This is not in the context of Borel graph theory, but assures that the assumption of Theorem \ref{intr1} is proper.
\end{rmk}

Now we explain the terminology. See section \ref{spre1} for more formal definitions. For a graph $(X,G)$, let 
\begin{equation*}
    d_G:X\times X\to \{0,1,2,...,\infty\}
\end{equation*}
be the (extended) path metric, and let $E_G$ be the equivalence relation generated by $G$, i.e.,
\begin{equation*}
    E_G=\{(x,y)\in X\times X\mid d_G(x,y)<\infty\}.
\end{equation*}

\begin{dfn}\label{intr3}
Let $(X,G)$ and $(Y,G')$ be graphs. An (\textit{extended}) \textit{quasi-isometry} $\gamma:(X,G)\to(Y,G')$ is a map $\gamma:X\to Y$ such that
\begin{enumerate}
    \item There exist $l\geq 1$ and $c\geq0$ such that
    \begin{equation*}
        l^{-1}d_G(x_1,x_2)-c\leq d_{G'}(\gamma(x_1),\gamma(x_2))\leq ld_G(x_1,x_2)+c
    \end{equation*}
     for all $x_1,x_2\in X$.
    \item We have
    \begin{equation*}
        \sup_{y\in Y} d_{G'}(y,\gamma(X))=\sup_{y\in Y}\inf_{x\in X}d_{G'}(y,\gamma(x))<\infty.
    \end{equation*}
\end{enumerate}
If a quasi-isometry $\gamma:(X,G)\to (Y,G')$ is a Borel map between Borel graphs, then it is called a \textit{Borel quasi-isometry}.
\end{dfn}

Note that condition (i) implies that $\gamma$ is a reduction of equivalence relations from $E_G$ to $E_{G'}$ since we have $d_G(x_1,x_2)<\infty$ if and only if $d_{G'}(\gamma(x_1),\gamma(x_2))<\infty$.

Let $\Delta_X=\{(x,x)\mid x\in X\}$ denote the diagonal set of $X\times X$.

\begin{dfn}\label{intr14}
For graphs $G$ and $H$ on a set $X$, the graph $G\cup H$ is denoted by $G\ast H$ if $E_G$ and $E_H$ are \textit{freely intersecting}, i.e., if a sequence $\{x_i\}_{i=0}^{2n}\subset X$ with $n\geq 1$ satisfies
\begin{equation*}
    (x_{2i},x_{2i+1})\in E_G\setminus \Delta_X \textup{ and }(x_{2i+1},x_{2i+2})\in E_H\setminus\Delta_X
\end{equation*}
for every $0\leq i\leq n-1$, then we have $x_0\neq x_{2n}$.
\end{dfn}

A \textit{cut} of a graph is a non-empty proper subset of an $E_G$-class whose edge boundary is finite.

\begin{dfn}\label{intr11}
Let $(X,G)$ be a graph with uniformly bounded degrees. The graph $G$ is said to be \textit{uniformly at most one-ended} if for every $k\geq 0$, there exists $r\geq 0$ such that for every cut $C$ of $G$ with $\textup{diam}_G\mleft(\partial_\textup{iv}^GC\mright)\leq k$, we have either $\textup{diam}_G(C)\leq r$ or $\textup{diam}_G\mleft(\overline{C}\mright)\leq r$. Here $\textup{diam}_G$ denotes the diameter with respect to $d_G$. The set $\partial_\textup{iv}^GC$ is the inner vertex boundary of the cut $C$, and $\overline{C}$ is the opposite-side cut of $C$.
\end{dfn}

\subsubsection*{The algebra and cohomology.} 
Let $R$ be a commutative ring.
\begin{nta}
For a set $X$, let $l^\infty_R(X)$ be the set of functions $X\to R$ whose images are finite.
\end{nta}

Let $(X,G)$ be a graph with uniformly bounded degrees.

\begin{dfn}\label{intr7}
For integers $k\geq 0$, set
\begin{align*}
    G^k&=\{(x,y)\in E_G \mid d_G(x,y)\leq k\} \textup{ and}\\
    R_G^k&=\left\{a\in l^\infty_R(E_G) \:\middle|\: a(E_G\setminus G^k)=0\right\}.
\end{align*}
Then the union $R_G=\bigcup_{k=0}^\infty R_G^k$ is a unital $R$-algebra with products defined by: for $a,b\in R_G$,
\begin{equation*}
    (a b)(x,y)=\sum_{z\in[x]_G}a(x,z)b(z,y),
\end{equation*}
where $[x]_G$ is the $E_G$-class containing $x$. This is well-defined since $G$ has uniformly bounded degrees.

The set $l^\infty_R(X)$ is identified with the subalgebra $R_G^0=l^\infty_R(\Delta_X)$, that is, $f\in l^\infty_R(X)$ is identified with the function $(x,x)\in \Delta_X\mapsto f(x)$. The function $1_X\in l^\infty_R(X)$ is the unit of $R_G$.
\end{dfn}

\begin{exm}\label{intr9}
Let $\Gamma$ be a group with a finite generating set $S\subset \Gamma\setminus\{e\}$. Suppose that $\Gamma$ acts freely on a set $X$. Consider the Schreier graph
\begin{equation*}
    G=\left\{(x,s^{\pm 1}x)\in X\times X \:\middle|\: x\in X,\ s\in S\right\}.
\end{equation*}
Then the algebra $R_G$ is isomorphic to the crossed product $R\Gamma\ltimes l^\infty_R(X)$ by the map
\begin{equation*}
    (\delta_\gamma,f)\in R\Gamma\ltimes l^\infty_R(X)\mapsto 1_{\{(\gamma x,x)\mid x\in X\}}\cdot f\in R_G.
\end{equation*}
\end{exm}

We define the cohomology of graphs, following the definition of cohomology of groups.

\begin{dfn}\label{intr6}
Let $M$ be a left $R_G$-module. For integers $n\geq 0$, define
\begin{equation*}
    \textup{H}^n(G,M)=\textup{Ext}_{R_G}^n(l^\infty_R(X),M).
\end{equation*}
Here $l^\infty_R(X)$ is regarded as a left $R_G$-module with the following module structure: for $a\in R_G$ and $f\in l^\infty_R(X)$, the function $a_* f\in l^\infty_R(X)$ is defined by
\begin{equation*}
    (a_* f)(x)=\sum_{z\in[x]_G}a(x,z)f(z).
\end{equation*}
\end{dfn}

\begin{rmk}
A left $R_G$-module can be regarded as a left $\mathbb{Z}_G$-module. However the cohomology groups $\textup{H}^n(G,M)$ do not depend on whether $M$ is regarded as a left $R_G$-module or as a left $\mathbb{Z}_G$-module (Lemma \ref{bk11}). Hence we do not need to indicate the coefficient $R$ in the notation $\textup{H}^n(G,M)$.
\end{rmk}

The cohomology group $\textup{H}^1(G,R_G)$ reflects the structure of cuts of $G$. In fact, we have $\textup{H}^1(G,R_G)=0$ if and only if $G$ is uniformly at most one-ended. Also, $\textup{H}^1(G,R_G)$ is finitely generated as a right $R_G$-module if and only if it is ``generated'' by a family of cuts of $G$ with uniformly bounded boundaries (Proposition \ref{intr10}).

\begin{rmk}
We explain the relationship with Roe's coarse cohomology \cite[Chapter 5]{Roe}. Note that in the above setting, the path metric $d_G$ defines the canonical coarse structure on $X$, so the coarse cohomology $\textup{HX}^n(X,R)$ can be defined. Then we have the isomorphism
\begin{equation*}
    \textup{HX}^n(X,R)\simeq \bigoplus_{x\in A}\textup{H}^n(G,R_G)\cdot 1_{\{x\}},
\end{equation*}
where $A\subset X$ is any representative set for the equivalence relation $E_G$. These cohomology groups have similar properties, but $\textup{H}^n(G,R_G)$ is involved with more uniform conditions than $\textup{HX}^n(X,R)$ in general. For example, $\textup{HX}^1(X,R)=0$ if and only if every component of $G$ is at most one-ended. As a result, if $(X,G)$ is a one-ended tree, then $\textup{HX}^1(X,R)=0$ but $\textup{H}^1(G,R_G)\neq 0$. Another important difference is that $\textup{H}^n(G,R_G)$ has a structure of right $R_G$-module while $\textup{HX}^n(X,R)$ does not admit such an action of translations. This enables us to discuss the ``finite generation" condition for $\textup{H}^1(G,R_G)$ in an appropriate sense.
\end{rmk}

\subsection{Borel graphs of cohomological dimension one}\label{sintr2}

Recall that for a commutative ring $R$, the $R$-\textit{cohomological dimension} of a group $\Gamma$, which is denoted by $\textup{cd}_R(\Gamma)$, is the smallest $n\in\{0,1,2,...,\infty\}$ such that $\textup{H}^i(\Gamma,M)=0$ for all left $R\Gamma$-module $M$ and $i>n$. One of the applications of the Stallings theorem for ends of groups is Stallings-Swan's theorem that for any non-zero commutative ring $R$, a torsion free group $\Gamma$ with $\textup{cd}_R(\Gamma)\leq 1$ must be free \cite[Theorem A]{Swa}. By using accessibility of groups, Dunwoody gives a variant of this theorem as follows:

\begin{thm}[{\cite[Corollary 1.2]{Dun1}}] \label{intr5}
Let $R$ be a non-zero divisible commutative ring and $\Gamma$ a finitely generated group. Then the following conditions are equivalent:
\begin{enumerate}
    \item The group $\Gamma$ is virtually free, i.e., it has a free subgroup of finite index.
    \item $\textup{cd}_R(\Gamma)\leq 1$.
\end{enumerate}
\end{thm}

\begin{rmk}
More generally, $R$ can be any non-zero commutative ring such that if $\Gamma$ has a torsion element of order $n$, then $n$ is a divisor of $1\in R$.
\end{rmk}

Now we define the cohomological dimension of graphs analogously.

\begin{dfn}\label{intr13}
Let $R$ be a commutative ring and $(X,G)$ a graph with uniformly bounded degrees. The $R$-\textit{cohomological dimension} of $G$, which is denoted by $\textup{cd}_R(G)$, is the smallest $n\in\{0,1,2,...,\infty\}$ such that $\textup{H}^i(G,M)=0$ for all left $R_G$-module $M$ and $i>n$.
\end{dfn}

Then as a consequence of Theorem \ref{intr1}, we obtain an analog of Theorem \ref{intr5} for Borel graphs as follows:

\begin{thmA}[Theorem \ref{mos4}] \label{intr2}
Let $R$ be a non-zero commutative ring and $(X,G)$ a Borel graph with uniformly bounded degrees. Then the following conditions are equivalent:
\begin{enumerate}
    \item There exists a Borel acyclic graph on $X$ Lipschitz equivalent to $G$.
    \item $\textup{cd}_R(G)\leq 1$.
\end{enumerate}
\end{thmA}

Here, we say that a graph $H$ on $X$ is \textit{Lipschitz equivalent} to $G$ if $\textup{id}_X:(X,H)\to (X,G)$ is a quasi-isometry. We should emphasize again that condition (i) ensures the existence of another Borel graph even though condition (ii) is only concerning the metric structure.

\subsubsection*{An application.}
Treeable equivalence relations have been studied as an analog of free groups. Chen-Poulin-Tao-Tserunyan show that treeability has a property similar to the quasi-isometric rigidity of free groups: If $(X,G)$ is a locally finite Borel graph such that every component is quasi-isometric to a tree, then the equivalence relation generated by $G$ is treeable \cite[Theorem 1.1]{CPTT}. Their proof applies the theory of median graphs to Borel graphs. Moreover, by refining the argument, they show the following:

\begin{thm}[{\cite[Theorem 1.2]{CPTT}}]\label{intr12}
Let $(X,G)$ be a Borel graph with uniformly bounded degrees. If $(X,G)$ is quasi-isometric to an acyclic graph as an abstract graph, then there exists an Borel acyclic graph on $X$ Lipschitz equivalent to $G$.
\end{thm}

Now we give an alternate proof of this theorem through Theorem \ref{intr2}: Let $R$ be a non-zero commutative ring. By assumption, there exists a (not necessarily Borel) acyclic graph $T$ on $X$ Lipschitz equivalent to $G$. Note that $R_G=R_T$ holds. Since $T$ is acyclic, it has $R$-cohomological dimension at most $1$ (Lemma \ref{bk4}), and so does $G$. Then by Theorem \ref{intr2}, there exists a Borel acyclic graph on $X$ Lipschitz equivalent to $G$. Theorem \ref{intr12} is proved.

This new proof has a common feature with the standard proof of the quasi-isometric rigidity of free groups. Indeed, our proof relies on the decomposition of Borel graphs given by Theorem \ref{intr1} in the same way that the proofs of 
\cite[Chapitre 7, Th\'{e}or\`{e}me 19]{GH} and \cite[Theorem 20.45]{DK} rely on accessibility of groups.

\begin{rmk}
Recently, Margolis generalizes Roe's coarse cohomology and defines the coarse $R$-cohomological dimension for metric spaces \cite{Mar}. For a bounded-degree connected graph $(X,G)$, it can be shown that the coarse $R$-cohomological dimension of $(X,d_G)$ is smaller than or equal to $\textup{cd}_R(G)$. We do not know whether these invariants coincide with each other. Now \cite[Proposition 8.10 and Theorem 10.2]{Mar} states that a quasi-geodesic space of coarse $R$-cohomological dimension at most $1$ it is quasi-isometric to a tree. This implies that if a bounded-degree graph $(X,G)$ satisfy $\textup{cd}_R(G)\leq 1$, then every component of $G$ is quasi-isometric to a tree. However, our proof of (ii)$\Rightarrow$(i) of Theorem \ref{intr2} does not use this result.
\end{rmk}

\begin{rmk}
It is proved in \cite[Theorem 1.4]{CPTT} that if $(X,G)$ is a locally finite Borel graph such that every component has \textit{bounded tree-width}, then the equivalence relation generated by $G$ is treeable. Jard\'{o}n-S\'{a}nchez also give a proof of this result in the case of uniformly bounded degrees, by using \textit{tree decompositions} \cite[Theorem 3]{Jar}. This method has some similarity to ours, and can be applied to the setting of Theorem \ref{intr12} since a bounded-degree quasi-tree has bounded tree-width. However, it is non-trivial to make the resulting Borel acyclic graph Lipschitz equivalent to $G$. See also Remark \ref{mos2}.
\end{rmk}

\subsection{Outline of the paper}

In section \ref{spre}, we prepare basic terminology and facts on graph theory and homological algebra.\par
In section \ref{slot}, we present a construction of the structure tree associated to a treeset, which is a main tool to prove our main theorem. Then we apply it to Borel graphs and obtain a prototype of the decomposition.\par 
In section \ref{sbk}, we first describe properties of the cohomology and the cohomological dimension of graphs. Then we investigate the group $\textup{H}^1(G,R_G)$, and its relation with the cutsets of $G$.\par 
In section \ref{smos}, by combining all the above content, we prove Theorems \ref{intr1} and \ref{intr2}.

\section{Preliminaries}\label{spre}

\subsection{Terminology of graph theory}\label{spre1}

\begin{dfn}
An (\textit{abstract}) \textit{graph} $(X,G)$ is a pair of the vertex set $X$ and the edge set $G\subset X\times X$ such that:
\begin{enumerate}
    \item $G\cap \Delta_X=\varnothing$, where $\Delta_X=\{(x,x)\mid x\in X\}$, \textup{ and}
    \item if $(x,y)\in G$, then $(y,x)\in G$.
\end{enumerate}
In this case, we also say that $G$ is a graph on $X$. If further $X$ is a standard Borel space and $G\subset X\times X$ is a Borel subset, then $(X,G)$ is called a \textit{Borel graph}.

If $G$ and $H$ are graphs on a set $X$ such that $H\subset G$, then $H$ is called a \textit{subgraph} of $G$.
\end{dfn}

\begin{dfn}
Let $(X,G)$ be a graph.
\begin{itemize}
    \item If a sequence $\{x_i\}_{i=0}^n$ satisfies $(x_i,x_{i+1})\in G$ for every $0\leq i\leq n-1$, then it is called a ($G$-)path from $x_0$ to $x_n$. The \textit{length} of this path is $n$. This path is \textit{simple} if $x_i\neq x_j$ for all $i\neq j$, and is a \textit{cycle} if $x_0=x_n$.
    \item The graph $(X,G)$ is \textit{connected} if for every $x,y\in X$, there exists a path from $x$ to $y$, and is \textit{acyclic} if there is no cycle $\{x_i\}_{i=0}^n$ with $n\geq 2$ such that the path $\{x_0\}_{i=0}^{n-1}$ is simple. A connected and acyclic graph is called a \textit{tree}.
    \item For a subset $A\subset X$, it is $G$-\textit{connected} if the graph $(A, (A\times A)\cap G)$ is connected. A maximal $G$-connected subset of $A$ is called a $G$-\textit{connected component} of $A$.
    \item The \textit{degree} of $x\in X$ is the number $|\{y\in X\mid (x,y)\in G\}|$. The graph $(X,G)$ has \textit{uniformly bounded degrees} if $\sup_{x\in X}|\{y\in X\mid (x,y)\in G\}|<\infty$ holds.
\end{itemize}
\end{dfn}

\begin{nta}
Let $(X,G)$ be a graph.
\begin{itemize}
    \item For $x,y\in X$, let $d_G(x,y)$ be the smallest length of the paths from $x$ to $y$ if such paths exist, and let $d_G(x,y)=\infty$ otherwise.
    \item The equivalence relation $E_G$ generated by $G$ is defined by
\begin{equation*}
    E_G=\{(x,y)\in X\times X\mid d_G(x,y)<\infty\}.
\end{equation*}
    For $x\in X$, the $E_G$-class containing $x$ is denoted by $[x]_G$. For $k\geq 0$, set
    \begin{equation*}
        G^k=\{(x,y)\mid d_G(x,y)\leq k\}.
    \end{equation*}
    \item For $x\in X,\ A\subset X$ and $k\geq 0$, set
\begin{align*}
    d_G(x,A)&=\inf_{y\in A}d_G(x,y),\\
    B_G(k;A)&=\{x\in X\mid d_G(x,A)\leq k\},\\
    B_G(k;x)&=B_G(k;\{x\})\textup{ and}\\
    \textup{diam}_G(A)&=\sup_{x,y\in A}d_G(x,y).
\end{align*}
    \item Let $\llangle G\rrangle$ be the set of bijections $\varphi:\textup{dom}\varphi\to\textup{im}\varphi$ between subsets of $X$ such that $\sup_{x\in \textup{dom}\varphi}d_G(\varphi x,x)<\infty$. The composition and the inverse on $\llangle G\rrangle$ are naturally defined. For $\varphi\in\llangle G\rrangle$, set
    \begin{equation*}
        \textup{graph}\varphi=\{(\varphi x,x)\in X\times X\mid x\in\textup{dom}\varphi\}.
    \end{equation*}
\end{itemize}
\end{nta}

\begin{rmk}
Let $(X,G)$ be a graph with uniformly bounded degrees. Then there exists a family $\{\varphi_i\}_{i=1}^n\subset\llangle G\rrangle$ such that $G=\bigsqcup_{i=1}^n\mleft(\textup{graph}\varphi_i\sqcup\textup{graph}\varphi_i^{-1}\mright)$.
\end{rmk}

Recall that quasi-isometries between graphs are defined in Definition \ref{intr3}.

\begin{rmk}
For a quasi-isometry $\gamma:(X,G)\to(Y,G')$, there exists a quasi-isometry $\lambda:(Y,G')\to (X,G)$ such that
\begin{align*}
    \sup_{x\in X}d_G(\lambda\circ\gamma(x),x)<\infty\ \textup{ and }
    \sup_{y\in Y}d_G(\gamma\circ\lambda(y),y)<\infty.
\end{align*}
This is called a \textit{quasi-isometric inverse} of $\gamma$. If $\gamma$ is a Borel quasi-isometry between Borel graphs with uniformly bounded degrees, then we can take a quasi-isometric inverse of $\gamma$ to be Borel by the Lusin-Novikov uniformization theorem. 
\end{rmk}

\begin{rmk}\label{pre4}
If $\gamma:(X,G)\to(Y,G')$ is an injective quasi-isometry, then there exists $l\geq 1$ such that $\gamma$ is $l$-biLipschitz, i.e.,
\begin{equation*}
        l^{-1}d_G(x_1,x_2)\leq d_{G'}(\gamma(x_1),\gamma(x_2))\leq ld_G(x_1,x_2)
\end{equation*}
for all $x_1,x_2\in X$. This follows from the fact that a bijective quasi-isometry between uniformly discrete metric spaces is automatically biLipschitz.
\end{rmk}

\begin{dfn}
Let $G$ and $H$ be graphs on a set $X$.
\begin{enumerate}
    \item We say that $H$ is a \textit{coarsely embedded subgraph} of $G$ if for every $k\geq 0$, there exists $l\geq 0$ such that if $H^k\subset G^l$ and $G^k\cap E_H\subset H^l$.
    \item We say that $G$ and $H$ are \textit{Lipschitz equivalent} if there exists $l\geq 1$ such that $H\subset G^l$ and $G\subset H^l$. When we indicate the number $l$, we say that they are $l$-Lipschitz equivalent.
\end{enumerate}
\end{dfn}

\begin{rmk}\label{pre1}
Let $R$ be a non-zero commutative ring and $G,H$ graphs on a set $X$ with uniformly bounded degrees. Then $H$ is a coarsely embedded subgraph of $G$ if and only if $R_H\subset R_G$ and $a|_{E_H}\in R_H$ for all $a\in R_G$. Also, $G$ and $H$ are Lipschitz equivalent if and only if $R_G=R_H$.
\end{rmk}

Let $(X,G)$ be a graph in this subsection below.

\begin{nta}\label{lot6}
Let $\omega\in X/E_G$ be an $E_G$-class. For a subset $C\subset \omega$, set
\begin{align*}
    \overline{C}&=\omega\setminus C,\\
    \partial_\textup{iv}^GC&=\left\{x\in C \:\middle|\: \exists y\in \overline{C},\ (x,y)\in G\right\},\\
    \partial_\textup{ov}^GC&=\partial_\textup{iv}^G\overline{C},\\
    \partial_\textup{ie}^GC&=(\overline{C}\times C)\cap G\textup{ and}\\
    \partial_\textup{oe}^GC&=\partial_\textup{ie}^G\overline{C}.\\
\end{align*}
\end{nta}

\begin{dfn}
A subset $C$ of an $E_G$-class is called a \textit{cut} of $G$ if $C$ and $\overline{C}$ are non empty and $\left|\partial_\textup{ie}^GC\right|=\left|\partial_\textup{oe}^GC\right|<\infty$. A family of cuts of $G$ is called a \textit{cutset} of $G$. Note that a cut $C$ of $G$ is determined only by $\partial_\textup{oe}^GC$.
\end{dfn}

\begin{dfn}
A cutset $\mathcal{C}$ of $G$ is said to 
\begin{enumerate}
    \item be \textit{nested} if for all cuts $C,D\in\mathcal{C}$ on the same $E_G$-class, either
    \begin{equation*}
        C\cap D,\ C\cap\overline{D},\ \overline{C}\cap D\textup{ or }\overline{C}\cap\overline{D}
    \end{equation*}
    is empty,
    \item be \textit{closed under complementation} if for every $C\in\mathcal{C}$, we have $\overline{C}\in\mathcal{C}$, and
    \item have \textit{uniformly bounded boundaries} if 
    \begin{equation*}
        \sup_{C\in\mathcal{C}}\textup{diam}_G\mleft(\partial_\textup{iv}^GC\mright)<\infty.
    \end{equation*}
\end{enumerate}
\end{dfn}

\begin{nta}\label{pre5}
Let $\textup{Cut}(G)$ be the set of cuts of $G$. There will be some situation that a family $\mathcal{C}$ is a subset of 
\begin{equation*}
    \textup{Cut}(G)\cup\{\omega\}_{\omega\in X/E_G}\cup\{\varnothing\}.
\end{equation*}
Then we set
\begin{equation*}
    \mathcal{C}^*=\mathcal{C}\cap\textup{Cut}(G).
\end{equation*}
\end{nta}

\subsection{Basic homological algebra}

Let $R$ be a (not necessarily commutative) ring. First we recall a construction of the Ext functor.

\begin{dfn}
An $R$-\textit{projective resolution} of a left $R$-module $N$ is an exact sequence
\begin{equation*}
    \cdots\rightarrow P_{n}\xrightarrow{\partial_n} P_{n-1} \to \cdots \xrightarrow{\partial_1} P_0\xrightarrow{\varepsilon} N\rightarrow 0
\end{equation*}
of left $R$-modules such that all $P_i$ are projective.
\end{dfn}

Let $N$ be a left $R$-module. Take an $R$-projective resolution
\begin{equation*}
    \cdots\rightarrow P_{n}\xrightarrow{\partial_n} P_{n-1} \to \cdots \xrightarrow{\partial_1} P_0\xrightarrow{\varepsilon} N\rightarrow 0.
\end{equation*}
For another left $R$-module $M$, the set of left $R$-homomorphisms from $N$ to $M$ is denoted by $\textup{Hom}_{R}(N,M)$. The cochain complex $\textup{Hom}_R(P_*,M)$ is defined by
\begin{equation*}
    \cdots\leftarrow \textup{Hom}_R(P_n,M)\xleftarrow{\partial_n^*} \textup{Hom}_R(P_{n-1},M)\leftarrow\cdots\xleftarrow{\partial_1^*} \textup{Hom}_R(P_0,M)\leftarrow 0,
\end{equation*}
where $\partial_n^*=-\circ \partial_n$.
Then for $n\geq 0$, the abelian group $\textup{Ext}_R^n(N,M)$ is defined by  the cohomology of this complex, i.e.,
\begin{equation*}
    \textup{Ext}_R^n(N,M)=\textup{H}^n(\textup{Hom}_R(P_*,M))
    =\left\{
\begin{array}{ll}
\textup{ker}\partial_{n+1}^*/\textup{im}\partial_n^*\quad (n\geq 1),\\
\textup{ker}\partial_1^*\quad (n=0).
\end{array}
\right.
\end{equation*}
These do not depend on the choice of projective resolutions since they are unique up to chain homotopy equivalence by the fundamental theorem of homological algebra.

Note that if $M$ is an $R$-$S$-bimodule for some ring $S$, then $\textup{Ext}_R^n(N,M)$ are right $S$-modules since $\textup{Hom}_R(P_*,M)$ is naturally a cochain complex of right $S$-modules.

\begin{prop}[{\cite[Proposition 8.6]{Rot}}]\label{pre2}
Let $N$ be a left $R$-module. For every integer $n\geq 0$, the following conditions are equivalent:
\begin{enumerate}
    \item $\textup{Ext}_R^i(N,M)=0$ for all left $R$-module $M$ and $i>n$.
    \item $\textup{Ext}_R^{n+1}(N,M)=0$ for all left $R$-module $M$.
    \item There exists an $R$-projective resolution of $N$ such that
    \begin{equation*}
        0\to P_{n}\to P_{n-1}\to\cdots\to P_0\to N\to 0.
    \end{equation*}
    \item For every $R$-projective resolution 
    \begin{equation*}
        \cdots\to P_{n}\xrightarrow{\partial_{n}} P_{n-1}\to\cdots\to P_0\xrightarrow{\partial_0} N\to 0
    \end{equation*}
    of $N$, the left $R$-module $\textup{im}\partial_n$ is projective.
\end{enumerate}
\end{prop}

The \textit{projective dimension} of a left $R$-module $N$ is the smallest $n\in\{0,1,2,...,\infty\}$ satisfying any of conditions (i)-(iv). It is denoted by $\textup{pd}_R(N)$.

\begin{rmk}
A left $R$-module $N$ is projective if and only if $\textup{pd}_R(N)=0$.
\end{rmk}

\begin{lem}\label{bk7}
Let $N$ be a left $R$-module and $n\geq 1$ an integer. If $\textup{pd}_R(N)\leq n$ and $\textup{Ext}_R^{n}(N,R)=0$, then we have $\textup{pd}_R(N)\leq n-1$.
\end{lem}

\begin{proof}
Let $M$ be a left $R$-module. Take a free $R$-module $F$ so that there exists a surjective homomorphism $q:F\to M$. Let 
\begin{equation*}
    0\to P_{n}\xrightarrow{\partial_n} P_{n-1} \to \cdots \to P_0\to N\to 0
\end{equation*}
be an $R$-projective resolution. Since $\textup{Ext}_R^{n}(N,R)=0$ holds and the functor $\textup{Ext}_R^n(N,-)$ preserves direct sums, we have $\textup{Ext}_R^n(N,F)=0$. This implies that the map $-\circ\partial_n:\textup{Hom}_{R}(P_{n-1}, F)\to \textup{Hom}_{R}(P_n,F)$ is surjective. Also the map $q\circ-:\textup{Hom}_{R}(P_n,F)\to \textup{Hom}_{R}(P_n,M)$ is surjective since $P_n$ is projective. Now the following diagram commutes
\begin{equation*}
    \begin{CD}
     \textup{Hom}_{R}(P_{n-1}, F) @>{-\circ\partial_{n}}>> \textup{Hom}_{R}(P_n,F) \\
      @V{q\circ-}VV                                                @V{q\circ-}VV \\
     \textup{Hom}_{R}(P_{n-1},M)             @>{-\circ\partial_{n}}>> \textup{Hom}_{R}(P_n,M)  .
    \end{CD}
\end{equation*}
Then since the upper $-\circ\partial_n$ and the right-side $q\circ -$ are surjective, the lower $-\circ\partial_n$ is also surjective. Hence we have $\textup{Ext}_R^n(N,M)=0$ for every left $R_G$-module $M$. This means that $\textup{pd}_R(N)\leq n-1$.
\end{proof}

\begin{lem}\label{pre3}
Suppose that a left $R$-module $N$ is projective and 
\begin{equation*}
    \cdots\to P_{n}\xrightarrow{\partial_n} P_{n-1} \to \cdots \to P_0\xrightarrow{\varepsilon} N\to 0
\end{equation*}
is an $R$-projective resolution. Then for any right $R$-module $M$, the sequence
\begin{equation*}
    \cdots\to M\otimes_RP_{n}\xrightarrow{\textup{id}_M\otimes\partial_n} M\otimes_RP_{n-1} \to \cdots \to M\otimes_RP_0\xrightarrow{\textup{id}_M\otimes\varepsilon} M\otimes_RN\to 0
\end{equation*}
is exact.
\end{lem}

\begin{proof}
By a construction of the Tor functor, we have
\begin{equation*}
    \textup{Tor}_n^R(M,N)=\textup{ker}(\textup{id}_M\otimes\partial_n)/\textup{im}(\textup{id}_M\otimes\partial_{n+1})
\end{equation*}
for $n\geq 1$, and these are $0$ since $N$ is projective. The exactness at $M\otimes_RP_0$ and $M\otimes_RN$ is trivial since the functor $M\otimes_R-$ is right exact.
\end{proof}

\section{Structure trees}\label{slot}

The goal of this section is to show Proposition \ref{lot8}.

\subsection{Construction of trees}\label{slot1}

In this subsection, let $X$ be a set, and for a subset $C\subset X$, set $\overline{C}=X\setminus C$. We say that two elements $x$ and $y$ of $X$ are \textit{separated} by a subset $C\subset X$ if either $(x\in C$ and $y\notin C)$ or $(y\in C$ and $x\notin C)$ holds.

\begin{dfn}\label{lot10}
A family $\mathcal{C}$ of non-empty proper subsets of $X$ is a \textit{treeset} on $X$ if it is
\begin{enumerate}
    \item nested, i.e., for all $C,D\in\mathcal{C}$, either 
    \begin{equation*}
        C\cap D,\ C\cap \overline{D},\ \overline{C}\cap D\ \textup{or }\overline{C}\cap \overline{D}
    \end{equation*}
    is empty,
    \item closed under complementation, i.e., if $C\in\mathcal{C}$, then $\overline{C}\in \mathcal{C}$ holds, and
    \item \textit{finitely separating}, i.e., for all $x,y\in X$, we have 
    \begin{equation*}
        |\{C\in\mathcal{C}\mid x\in C,\ y\notin C\}|<\infty.
    \end{equation*}
\end{enumerate}
\end{dfn}

The \textit{structure tree} associated to a treeset is introduced by \cite{Dun1}. We will present a construction of it, which is based on \cite[Chapter II]{DD} but described from the viewpoint of ultrafilters in the sense of \cite{Rol}. The proofs are given for the reader's convenience.

Let $\mathcal{C}$ be a treeset on $X$.

\begin{dfn}\label{tre2}
Let $V_\mathcal{C}$ be the set of subsets $u\subset \mathcal{C}$ such that:
\begin{enumerate}
    \item[(U1)] For every $C\in\mathcal{C}$, we have $\left|u\cap\left\{C,\overline{C}\right\}\right|=1$.
    \item[(U2)] If $C\in u$ and $C\subset D\in\mathcal{C}$, then $D\in u$ holds.
    \item[(U3)] There is no strictly decreasing sequence $C_0\supsetneq C_1\supsetneq\cdots$ in $u$.
\end{enumerate}
Then let
\begin{equation*}
    T_\mathcal{C}=\{(u,v)\in V_\mathcal{C}\times V_\mathcal{C}\mid |u\setminus v|=1\}.
\end{equation*}
This is a graph on $V_\mathcal{C}$ since $|u\setminus v|=1$ implies $|v\setminus u|=1$ by condition (U1).
\end{dfn}

Now we will show the following:

\begin{prop}\label{tre3}
The graph $(V_\mathcal{C},T_\mathcal{C})$ is a tree.
\end{prop}

\begin{lem}\label{tre7}
For $u,v\in V_{\mathcal{C}}$, the set $u\setminus v$ is finite and totally ordered.
\end{lem}

\begin{proof}
Let $C,D\in u\setminus v$. Since $\mathcal{C}$ is nested, we have either 
\begin{equation*}
    C\subset D,\ C\subset  \overline{D},\ \overline{C}\subset D\ \textup{or }\overline{C}\subset \overline{D}.
\end{equation*}
If $C\subset \overline{D}$, then we have $\overline{D}\in u$ by condition (U2), but we also have $D\in u$, which contradicts condition (U1). If $\overline{C}\subset D$, then we have $D\in v$, which is a contradiction again. Hence we have either $C\subset D$ or $D\subset C$, and thus $u\setminus v$ is totally ordered.

Suppose that $|u\setminus v|=\infty$. Then there exists either a strictly decreasing sequence or a strictly increasing sequence in $u\setminus v$. This contradicts condition (U3), and thus $u\setminus v$ is finite.
\end{proof}

\begin{lem}\label{tre5}
Let $u,v\in V_\mathcal{C}$. If $C$ is minimal in $u\setminus v$, then it is minimal in $u$. In particular, for $(u,v)\in T_\mathcal{C}$, the unique element of $u\setminus v$ is minimal in $u$.
\end{lem}

\begin{proof}
Let $C$ be minimal in $u\setminus v$, and let $C\supsetneq D\in\mathcal{C}$. Then we have $\overline{D}\in v$ since $\overline{D}\supset\overline{C}\in v$ holds. If $D\in u$, then we have $D\in u\setminus v$, which contradicts that $C$ is minimal in $u\setminus v$. Hence we have $D\notin u$, and thus $C$ is minimal in $u$.
\end{proof}

\begin{lem}\label{tre6}
If $C_0$ is minimal in $v\in V_\mathcal{C}$, then we have $v\bigtriangleup \left\{C_0,\overline{C_0}\right\}\in V_\mathcal{C}$.
\end{lem}

\begin{proof}
Let $C_0$ be minimal in $v\in V_\mathcal{C}$, and set $u=v\bigtriangleup \left\{C_0,\overline{C_0}\right\}$. It is clear that $u$ satisfies conditions (U1) and (U3).

Suppose that $C\in u$ and $C\subsetneq D\in\mathcal{C}$. If $C\in v$, then we have $D\in v$, and also have $D\neq C_0$ since $C_0$ is minimal in $v$. If $C=\overline{C_0}$, then we have $\overline{D}\subsetneq C_0$, which implies that $D\in v$ by the minimality of $C_0$. Hence we have $D\in u$ in any case, and condition (U2) is proved.
\end{proof}

\begin{lem}\label{tre8}
For every $C\in\mathcal{C}$, there exists a unique $(u,v)\in T_\mathcal{C}$ such that $u\setminus v=\{C\}$.
\end{lem}

\begin{proof}
Let $C\in\mathcal{C}$. By Lemmas \ref{tre5} and \ref{tre6}, it suffices to show that there exists a unique $u\in V_\mathcal{C}$ such that $C$ is minimal in $u$. We set
\begin{equation*}
    u=\left\{D\in\mathcal{C} \:\middle|\: C\subset D\textup{ or }\overline{C}\subsetneq D\right\}.
\end{equation*}
Then $u$ satisfies the conditions (U1) and (U2). Indeed, the former follows from that $\mathcal{C}$ is nested, and the latter is clear. Now let $D_0\supsetneq D_1\supsetneq\cdots$ be a strictly decreasing sequence in $u$. Then we have either $C\subset D_n$ for every $n$, or $\overline{C}\subsetneq D_n$ for every $n$. We assume the former case, and take $x\in C$ and $y\in \overline{D_0}$. Then $x\in D_n$ and $y\notin D_n$ hold for every $n$, which contradicts that $\mathcal{C}$ is finitely separating. In the same way, we can deduce a contradiction in the latter case. Hence $u$ satisfies condition (U3), and thus $u\in V_\mathcal{C}$.

Now $C$ is minimal in $u$. Indeed, if $C\supsetneq D\in \mathcal{C}$, then we have $\overline{C}\subsetneq \overline{D}$ and thus $\overline{D}\in u$. Finally, we will verify the uniqueness. Suppose that $C$ is minimal in $v\in V_\mathcal{C}$. For $D\in \mathcal{C}$, if $C\subset D$, then $D\in v$ holds. If $\overline{C}\subsetneq D$, then we have $D\in v$ by $\overline{D}\subsetneq C$ and the minimality of $C$. Hence we have $u\subset v$, which implies $u=v$ since both satisfy condition (U1). 
\end{proof}

\begin{proof}[Proof of Proposition \ref{tre3}]
First, we show that the graph $(V_\mathcal{C}, T_\mathcal{C})$ is connected. Let $u,v\in V_\mathcal{C}$. We prove that there exists a $T_\mathcal{C}$-path from $u$ to $v$ by the induction on $|u\setminus v|$. We may assume $u\neq v$ and let $C$ be the minimal element of $u\setminus v$, which exists by Lemma \ref{tre7}. Then by Lemmas \ref{tre5} and \ref{tre6}, $u'=u\bigtriangleup\{C,\overline{C}\}$ satisfies $(u,u')\in T_\mathcal{C}$ and $|u'\setminus v|=|u\setminus v|-1$. By the induction hypothesis, there exists a $T_\mathcal{C}$-path from $u'$ to $v$, and thus the claim is proved.

Next, we show that $(V_\mathcal{C}, T_\mathcal{C})$ is acyclic. Let $\{u_i\}_{i=0}^n$ be a cycle of $T_\mathcal{C}$ with $n\geq 2$ such that the path $\{u_i\}_{i=0}^{n-1}$ is simple. Let $u_0\setminus u_1=\{C\}$. Then by Lemma \ref{tre8}, for $1\leq i\leq n-1$, we have $u_{i+1}\setminus u_i\neq \{C\}$ since $(u_{i+1},u_i)\neq (u_0,u_1)$. This implies that $C\notin u_i$ for $1\leq i\leq n$ by the induction on $i$, which contradicts that $u_n=u_0$. Hence this graph is a tree.
\end{proof}

\begin{rmk}\label{tre4}
By the above proof, we have $d_{T_\mathcal{C}}(u,v)=|u\setminus v|$ for $u,v\in V_\mathcal{C}$.
\end{rmk}

Finally we will define a map $\rho:X\rightarrow V_\mathcal{C}$.

\begin{prop}\label{tre9}
For $x\in X$, let $\rho(x)=\{C\in\mathcal{C}\mid x\in C\}$. Then the following conditions hold:
\begin{enumerate}
    \item $\rho(x)\in V_\mathcal{C}$ for every $x\in X$.
    \item $d_{T_\mathcal{C}}(\rho(x),\rho(y))=|\{C\in\mathcal{C}\mid x\in C,\ y\notin C\}|$ for every $x,y\in X$.
\end{enumerate}
\end{prop}

\begin{proof}
(i) It is clear that $u=\rho(x)$ satisfies conditions (U1) and (U2). Let $C_0\supsetneq C_1\supsetneq\cdots$ be a strictly decreasing sequence in $\rho(x)$. Take $y\in \overline{C_0}$. Then we have $x\in C_n$ and $y\notin C_n$ for every $n$, which contradicts that $\mathcal{C}$ is finitely separating. Hence we have $\rho(x)\in V_\mathcal{C}$.

(ii) This follows from Remark \ref{tre4}.
\end{proof}

\subsection{Borel cutsets of Borel graphs}

Let $(X,G)$ be a Borel graph with uniformly bounded degrees. Set
\begin{equation}
    d=\sup_{x\in X}|\{y\in X\mid (x,y)\in G\}|. \label{eqtre3}
\end{equation}
Let $\textup{Cut}(G)$ be the set of cuts of $G$. This is identified with the set
\begin{equation*}
    \left\{\partial_\textup{oe}^GC \:\middle|\: C\in\textup{Cut}(G) \right\},
\end{equation*}
which is a Borel subset of the set of finite subsets of $G$. Hence $\textup{Cut}(G)$ is regarded as a standard Borel space.

\begin{dfn}
A Borel subset of $\textup{Cut}(G)$ is called a \textit{Borel cutset} of $G$. A Borel cutset $\mathcal{C}$ of $G$ is called a \textit{Borel treeset} of $G$ if
\begin{equation*}
    \mathcal{C}_\omega=\{C\subset\omega\mid C\in\mathcal{C}\}
\end{equation*}
is a treeset on $\omega$ for every $E_G$-class $\omega$.
\end{dfn}

\begin{lem}\label{lot4}
Let $\mathcal{C}$ be a Borel cutset of $G$ with uniformly bounded boundaries, i.e.,
\begin{equation*}
    r:=\sup_{C\in\mathcal{C}}\textup{diam}_G\mleft(\partial_\textup{iv}^G C\mright)<\infty.
\end{equation*}
Then we have
\begin{equation*}
    \sup_{x\in X}\left|\left\{C\in\mathcal{C} \:\middle|\:  x\in\partial_\textup{iv}^GC\right\}\right|<\infty.
\end{equation*}
\end{lem}

\begin{proof}
Fix $x\in X$. If $C\in \mathcal{C}$ and $x\in\partial_\textup{iv}^GC$, then $\partial_\textup{oe}^GC\subset \mleft(B_G(r;x)\times X\mright)\cap G$. By equation \eqref{eqtre3}, we have
\begin{equation*}
    |B_G(r;x)|\leq \sum_{i=0}^r d^i\leq d^{r+1}
\end{equation*}
and thus $|\mleft(B_G(r;x)\times X\mright)\cap G|\leq d\cdot d^{r+1}= d^{r+2}$.
Hence we have
\begin{equation*}
    \left|\left\{C\in\mathcal{C} \:\middle|\: x\in\partial_\textup{iv}^GC\right\}\right|\leq \left|\left\{\partial_\textup{oe}^GC\subset\mleft(B_G(r;x)\times X\mright)\cap G \:\middle|\:  C\in\mathcal{C}\right\}\right|\leq 2^{d^{r+2}},
\end{equation*}
and the lemma is proved.
\end{proof}

\begin{lem}\label{lot1}
If a Borel cutset $\mathcal{C}$ of $G$ with uniformly bounded boundaries is nested and closed under complementation, then it is a Borel treeset of $G$.
\end{lem}

\begin{proof}
It suffices to show that for $(x,y)\in E_G$, we have
\begin{equation*}
    |\{C\in\mathcal{C}\mid x\in C,\ y\notin C\}|<\infty.
\end{equation*}
Take a $G$-path $\{x_i\}_{i=0}^n$ from $x$ to $y$. If $C\in\mathcal{C}$ satisfies $x\in C$ and $y\notin C$, then we have $x_i\in \partial_\textup{iv}^GC$ for some $0\leq i\leq n-1$. By the previous lemma, there exist only finitely many $C\in\mathcal{C}$ satisfying this property. Hence the lemma is proved.
\end{proof}

The following is inspired by the proof of \cite[Proposition 3.3]{CPTT}:

\begin{lem}\label{lot7}
Let $\mathcal{C}$ be a Borel cutset of $G$ with uniformly bounded boundaries closed under complementation. Then there exist finitely many Borel treesets $\mathcal{C}_i\ (i=0,1,...,n)$ of $G$ such that $\mathcal{C}=\bigsqcup_{i=0}^n\mathcal{C}_i$.
\end{lem}

\begin{proof}
Set 
\begin{equation*}
    r=\sup_{C\in\mathcal{C}}\textup{diam}_G\mleft(\partial_\textup{iv}^GC\cup\partial_\textup{ov}^GC\mright).
\end{equation*}
Let $C,D\in\mathcal{C}$ be cuts on the same $E_G$-class, and let $x\in \partial_\textup{iv}^GC\cup\partial_\textup{ov}^GC$ and $y\in\partial_\textup{iv}^GD\cup\partial_\textup{ov}^GD$. We claim that if
\begin{equation*}
    B_G(r;x)\cap B_G(r;y)=\varnothing,
\end{equation*}
then $C$ and $D$ are nested with each other. Indeed, since $B_G(r;y)$ is connected and does not intersect $\partial_\textup{iv}^GC\cup\partial_\textup{ov}^GC\subset B_G(r;x)$, we have $B_G(r;y)\subset C$ or $B_G(r;y)\subset \overline{C}$. By replacing $C$ by $\overline{C}$ if needed, we may assume that $B_G(r;y)\subset \overline{C}$. Then since $C\cup B_G(r;x)$ is connected and does not intersect $\partial_\textup{iv}^GD\cup\partial_\textup{ov}^GD\subset B_G(r;y)$, we have $C\subset D$ or $C\subset \overline{D}$, and thus $C$ and $D$ are nested with each other.

Hence the number of cuts in $\mathcal{C}$ non-nested with $C$ is bounded by
\begin{align*}
    &\sum_{y\in B_G(2r;x)}\left|\left\{D\in\mathcal{C}\:\middle|\:  y\in\partial_\textup{iv}^GD\cup\partial_\textup{ov}^GD\right\}\right|\\
    &\leq|B_G(2r;x)|\cdot2\sup_{y\in X}\left|\left\{D\in\mathcal{C}\:\middle|\:   y\in\partial_\textup{iv}^GD\right\}\right| \notag \\ 
    &\leq 2d^{2r+1}\sup_{y\in X}\left|\left\{D\in\mathcal{C}\:\middle|\:  y\in\partial_\textup{iv}^GD\right\}\right|=:N .
\end{align*}
Take Borel subset $\mathcal{C}^+\subset\mathcal{C}$ so that $\left|\left\{C,\overline{C}\right\}\cap\mathcal{C}^+\right|=1$ for every $C\in\mathcal{C}$. By \cite[Proposition 4.6]{KST}, the Borel graph
\begin{equation*}
    \bigsqcup_{\omega\in X/E_G} \left\{(C,D)\in \mathcal{C}^+_\omega\times\mathcal{C}^+_\omega \:\middle|\: C\textup{ and }D\textup{ are non-nested with each other} \right\}
\end{equation*}
admits an $(N+1)$-Borel coloring $\mathcal{C}^+=\bigsqcup_{i=0}^N \mathcal{C}^+_i$, that is, each $\mathcal{C}^+_i$ is a nested Borel cutset. Then $\mathcal{C}_i=\left\{C,\overline{C}\mid C\in\mathcal{C}^+_i\right\}$ is a Borel nested cutset with uniformly bounded boundaries, closed under complementation. Hence these are Borel treesets by Lemma \ref{lot1}, and satisfy $\mathcal{C}=\bigcup_{i=0}^n\mathcal{C}_i$.
\end{proof}

Recall Definition \ref{intr3} and Notation \ref{pre5}.

\begin{lem}\label{lot11}
Let $\mathcal{C}$ be a Borel cutset of $G$ with uniformly bounded boundaries. 
\begin{enumerate}
    \item If $(Y,G')$ is a Borel graph with uniformly bounded degrees and $\lambda:(Y,G')\to (X,G)$ is a Borel quasi-isometry, then ${\lambda^{-1}(\mathcal{C})}^*$ is a Borel cutset of $G'$ with uniformly bounded boundaries, where
    \begin{equation*}
    {\lambda^{-1}(\mathcal{C})}=\left\{\lambda^{-1}(C)\mid C\in\mathcal{C}\right\}.
    \end{equation*}
    \item If $H$ is a subgraph of $G$, then $\mathcal{C}{|_H}^*$ is a Borel cutset of $G$ with uniformly bounded boundaries, where
    \begin{equation*}
        \mathcal{C}{|_H}=\{C\cap \omega\mid C\in\mathcal{C},\ \omega\in X/E_H\}.
    \end{equation*}
\end{enumerate}
\end{lem}

\begin{proof}
(i) Take $l\geq 1$ so that $(\lambda\times\lambda)(G')\subset G^l$. For a cut $C$ of $G$, let $(y_1,y_2)\in \partial_\textup{oe}^{G'}\lambda^{-1}(C)$. Then we have $d_G(\lambda(y_1),\lambda(y_2))\leq l$ and thus $\lambda(y_1)\in B_G\mleft(l;\partial_\textup{iv}^G{C}\mright)$. This implies that $\partial_\textup{iv}^{G'}\lambda^{-1}(C)\subset \lambda^{-1}\mleft(B_G\mleft(l;\partial_\textup{iv}^G{C}\mright)\mright)$. Hence ${\lambda^{-1}(\mathcal{C})}^*$ is a cutset of $G'$ with uniformly bounded boundaries. Also, the map
\begin{equation*}
    \{C\in\mathcal{C}\mid \lambda^{-1}(C)\in\textup{Cut}(G')\}\to \textup{Cut}(G'),\ C\mapsto \lambda^{-1}(C)
\end{equation*}
is Borel and finite-to-one. Then the image is ${\lambda^{-1}(\mathcal{C})}^*$, which is a Borel subset of $\textup{Cut}(G')$.

(ii) It is clear that $\mathcal{C}{|_H}^*$ is a cutset of $G$ with uniformly bounded boundaries. Also, the map
\begin{equation*}
    \{(C,x)\in\mathcal{C}\times X\mid C\cap [x]_H\in \textup{Cut}(H)\}\mapsto C\cap [x]_H
\end{equation*}
is Borel and countable-to-one. Then the image is $\mathcal{C}{|_H}^*$, which is a Borel subset of $\textup{Cut}(H)$.
\end{proof}

\subsection{Structure trees for Borel graphs}\label{slot3}

Let $(X,G)$ be a Borel graph with uniformly bounded degrees and $\mathcal{C}$ a Borel treeset of $G$. Then for every $\omega\in X/E_G$, we have the treeset $\mathcal{C}_\omega$ on $\omega$, the tree $(V_{\mathcal{C}_\omega},T_{\mathcal{C}_\omega})$ and the map $\rho_\omega:\omega\to V_{\mathcal{C}_\omega}$ by Propositions \ref{tre3} and \ref{tre9}. We set
\begin{equation*}
    (V_\mathcal{C},T_\mathcal{C})=\bigsqcup_{\omega\in X/E_G}(V_{\mathcal{C}_\omega},T_{\mathcal{C}_\omega})
    \textup{ and }\rho=\bigsqcup_{\omega\in X/E_G}\rho_\omega:X\to V_{\mathcal{C}}.
\end{equation*}
Then $(V_\mathcal{C},T_\mathcal{C})$ might not be a Borel graph, but $T_\mathcal{C}$ admits a standard Borel structure since it is identified with the Borel set $\mathcal{C}$ by Lemma \ref{tre8}.

\begin{lem}
The function $(x,y)\in E_G\mapsto d_{T_\mathcal{C}}(\rho(x),\rho(y))$ is Borel. In particular, $(\rho\times\rho)^{-1}(\Delta_{V_{\mathcal{C}}})$ is a Borel equivalence relation on $X$.
\end{lem}

\begin{proof}
This follows from the equation
\begin{equation*}
    d_{T_\mathcal{C}}(\rho(x),\rho(y))=|\{C\in\mathcal{C}\mid x\in C,\ y\notin C\}|
\end{equation*}
for every $(x,y)\in E_G$ given by Proposition \ref{tre9} (ii).
\end{proof}

\begin{lem}\label{tre1}
Suppose that for every $(x,y)\in G$, we have
\begin{equation*}
    |\{C\in\mathcal{C}\mid x\in C,\ y\notin C\}|\leq 1.
\end{equation*}
Then $\rho:(X,G)\rightarrow (V_\mathcal{C},T_\mathcal{C})$ is a $1$-surjective simplicial map, i.e.,
\begin{equation*}
    T_\mathcal{C}\subset (\rho\times\rho)(G)\subset \Delta_{V_\mathcal{C}}\cup T_\mathcal{C}.
\end{equation*}
\end{lem}

\begin{proof}
For every $(x,y)\in G$, we have $|\rho(x)\setminus\rho(y)|\leq 1$ by assumption, which implies that $\rho(x)=\rho(y)$ or $(\rho(x),\rho(y))\in T_\mathcal{C}$. Hence $(\rho\times\rho)(G)\subset \Delta_{V_\mathcal{C}}\cup T_\mathcal{C}$ holds. 

For $(u,v)\in T_\mathcal{C}$, take $C\in\mathcal{C}$ so that $u\setminus v=\{C\}$. Then any $(x,y)\in \partial_\textup{oe}^GC$ satisfies $\rho(x)\setminus\rho(y)=\{C\}$ and thus $(\rho(x),\rho(y))=(u,v)$. Hence $T_\mathcal{C}\subset (\rho\times\rho)(G)$ holds.
\end{proof}

The following is very inspired by the proof of \cite[Theorem 4.1]{Tse}. 

\begin{prop}\label{lot8}
Let $\mathcal{C}$ be a Borel treeset of $G$ with uniformly bounded boundaries such that for every $(x,y)\in G$,
\begin{equation*}
    |\{C\in\mathcal{C}\mid x\in C,\ y\notin C\}|\leq 1.
\end{equation*}
Then there exists a Borel graph $G'$ on $X$ Lipschitz equivalent to $G$ such that:
\begin{enumerate}
    \item $G'=T\ast H$ with $T$ and $H$ Borel subgraphs of $G'$. 
    \item $T$ is acyclic.
    \item $\mathcal{C}{|_H}^*=\varnothing$.
\end{enumerate}
\end{prop}

\begin{proof}
By the previus lemma, $\rho:(X,G)\rightarrow (V_\mathcal{C},T_\mathcal{C})$ is a $1$-surjective simplicial map. Note that the map
\begin{equation*}
    \rho\times\rho:(\rho\times\rho)^{-1}(T_\mathcal{C})\cap G\to T_\mathcal{C}
\end{equation*}
is a finite-to-one surjective Borel map. Indeed, for $C\in\mathcal{C}$, this map sends the elements of $\partial_\textup{oe}^GC$ to the unique edge of $T_\mathcal{C}$ identfied with $C$ by Lemma \ref{tre8}. Hence we can take a Borel subgraph $T\subset G$ so that $(\rho\times\rho)|_T:T\to T_\mathcal{C}$ is bijective. Then $T$ is clearly acyclic.

Set
\begin{equation*}
    G_1=G\cup\mleft(\bigcup_{C\in\mathcal{C}}\mleft(\partial_\textup{iv}^GC\times \partial_\textup{iv}^GC\mright)\setminus\Delta_X\mright)
\end{equation*}
Note that for $C\in\mathcal{C}$, the map $\rho\times\rho$ sends all elements of $\partial_\textup{oe}^GC$ to a single edge of $T_\mathcal{C}$, and thus $\rho$ sends all elements of $\partial_\textup{iv}^GC$ to a single vertex. Hence $\rho:(X,G_1)\rightarrow (V_\mathcal{C},T_\mathcal{C})$ is also a $1$-surjective simplicial map. Now set
\begin{equation*}
    H=(\rho\times\rho)^{-1}(\Delta_{V_\mathcal{C}})\cap G_1.
\end{equation*}
Then it is clear that $\mathcal{C}{|_H}^*=\varnothing$ since the endpoints of every edge of $H$ is not separated by any cut in $\mathcal{C}$.

To show that $E_T$ and $E_H$ are freely intersecting, let $n\geq 1$ and $\{x_i\}_{i=0}^{2n}$ be a sequence of $X$ such that
\begin{equation*}
    (x_{2i},x_{2i+1})\in E_T\setminus\Delta_X\textup{ and } (x_{2i+1},x_{2i+2})\in E_H\setminus\Delta_X
\end{equation*}
for every $0\leq i\leq n-1$. It suffices to verify that $x_0\neq x_{2n}$. Note that $\rho(x_{2i})\neq \rho(x_{2i+1})=\rho(x_{2i+2})$ for every $0\leq i\leq n-1$. We prove that
\begin{equation}
    d_{T_\mathcal{C}}(\rho(x_0),\rho(x_{2n}))=\sum_{i=0}^{n-1}d_{T_\mathcal{C}}(\rho(x_{2i}),\rho(x_{2i+2})) \label{eqlot2}
\end{equation}
by the induction on $n$. We may assume $n\geq 2$. Let $\{y_j\}_{j=0}^k$ be the simple $T$-path from $x_0$ to $x_1$, and let $\{z_j\}_{j=0}^l$ be the simple $T$-path from $x_2$ to $x_3$. Then $\{\rho(y_j)\}_{j=0}^k$ and $\{\rho(z_j)\}_{j=0}^l$ are simple $T_\mathcal{C}$-paths since $(\rho\times\rho)|_{T}:T\rightarrow T_\mathcal{C}$ is bijective. Moreover, we have $(y_k,y_{k-1})\neq (z_0,z_1)$ since $y_k=x_1\neq x_2=z_0$. This implies that $(\rho(y_k),\rho(y_{k-1}))\neq (\rho(z_0),\rho(z_1))$, and thus 
\begin{equation*} 
    \rho(y_0),\rho(y_1)...,\rho(y_k)=\rho(z_0),\rho(z_1),...,\rho(z_l)
\end{equation*}
is a simple $T_\mathcal{C}$-path from $\rho(x_0)$ to $\rho(x_3)=\rho(x_4)$ passing through $\rho(x_1)=\rho(x_2)$. Then we have
\begin{equation*}
    d_{T_\mathcal{C}}(\rho(x_0),\rho(x_{4}))=d_{T_\mathcal{C}}(\rho(x_0),\rho(x_{2}))+d_{T_\mathcal{C}}(\rho(x_2),\rho(x_{4})),
\end{equation*}
and equation \eqref{eqlot2} holds by the induction hypothesis. In particular, we have $d_{T_\mathcal{C}}(\rho(x_0),\rho(x_{2n}))>0$ and thus $x_0\neq x_{2n}$.

Finally we show $G'=T\ast H$ is Lipschitz equivalent to $G$. Set 
\begin{equation*}
    r=\sup_{C\in\mathcal{C}}\textup{diam}_G\mleft(\partial_\textup{iv}^G C\mright).
\end{equation*}
Then for every $(x,y)\in G_1\setminus G$, we have $d_G(x,y)\leq r$ since $x,y\in \partial_\textup{iv}^G C$ for some $C\in\mathcal{C}$. Hence $G$ and $G_1$ are $r$-Lipschitz equivalent. For every $(x,y)\in G_1\setminus H$, we have $(\rho(x),\rho(y))\in T_\mathcal{C}$. Then there exists a unique $(x',y')\in T$ such that $(\rho(x'),\rho(y'))=(\rho(x),\rho(y))$. We can take $C\in\mathcal{C}$ so that $(x,y),(x',y')\in\partial_\textup{oe}^GC$, which implies that
\begin{align*}
    &(x,x')\in\partial_\textup{iv}^GC\times \partial_\textup{iv}^GC\subset H\cup\Delta_X\textup{ and }\\&(y,y')\in\partial_\textup{iv}^G\overline{C}\times \partial_\textup{iv}^G\overline{C}\subset H\cup\Delta_X.
\end{align*}
Hence we have $d_{G'}(x,y)\leq 3$, and $G_1$ and $G'$ are $3$-Lipschitz equivalent. Eventually $G$ and $G'$ are $3r$-Lipschitz equivalent.
\end{proof}

\section{Cohomology of graphs}\label{sbk}

Througout this section, let $R$ be a non-zero commutative ring and $(X,G)$ an abstract graph with uniformly bounded degrees. Many discussions here are analogous to the usual group cohomology (see e.g., \cite{Co}).

\subsection{Computaion of cohomology}\label{sbk1}

Recall that $R_G$ is an $R$-algebra associated to $G$, and $l^\infty_R(X)$ is identified with the subalgebra $l^\infty_R(\Delta_X)\subset R_G$ (Definition \ref{intr7}).

\begin{nta}
For $\varphi\in\llangle G\rrangle$, set $\alpha_\varphi=1_{\textup{graph}\varphi}\in R_G$.
Note that
\begin{align*}
    \alpha_\varphi\alpha_\psi=\alpha_{\varphi\psi}\textup{ and }\alpha_\varphi f\alpha_{\varphi^{-1}}=\alpha_{\varphi*}f
\end{align*}
hold for $\varphi,\psi\in\llangle G\rrangle$ and $f\in l^\infty_R(X)$, where $\alpha_{\varphi*}f$ is the multiplication in the left $R_G$-module $l^\infty_R(X)$ as defined in Definition \ref{intr6}.
\end{nta}

\begin{rmk}\label{bk16}
If $G=\bigcup_{i=1}^n\textup{graph}\varphi_i$ with $\{\varphi_i\}_{i=1}^n\subset\llangle G\rrangle$, then $R_G$ is generated by the set $\{\alpha_{\varphi_i}\}_{i=1}^n\cup l^\infty_R(X)$ as a ring. Indeed, for every $k\in\mathbb{Z}_{\geq 0}$, there exists $\{\psi_j\}_{j=1}^m\subset\llangle G\rrangle$ such that $G^k=\bigcup_{j=1}^m\textup{graph}\psi_j$ and each $\psi_j$ is obtained by composing elements of $\{\varphi_i\}_{i=1}^n$ at most $k$ times. Then for every $a\in R_G^k$, there exists $\{f_j\}_{j=1}^m\subset l^\infty_R(X)$ such that
\begin{equation*}
    a=\sum_{j=1}^m \alpha_{\psi_j}f_j=\sum_{j=1}^m \mleft(\alpha_{\psi_j*}f_j\mright)\alpha_{\psi_j}.
\end{equation*}
\end{rmk}

\begin{rmk}\label{bk17}
For a subset $A\subset X$, the left $R_G$-module $R_G\cdot 1_A$ is projective. Indeed, this is the set of functions in $R_G$ supported on $X\times A$, and thus $R_G=(R_G\cdot1_A)\oplus (R_G\cdot1_{X\setminus A})$. Hence $R_G\cdot1_A$ is a direct summand of $R_G$.

In the same way, the left $l^\infty_R(X)$-module $l^\infty_R(X)\cdot 1_A$ is projective.
\end{rmk}

\begin{rmk}\label{bk12}
The algebra $R_G$ is projective as a left $l^\infty_R(X)$-module. Indeed, by Remark \ref{bk16}, we can take $\{\varphi_i\}_{i=1}^\infty\subset \llangle G\rrangle$ so that for every $k\in\mathbb{Z}_{\geq 0}$, there exists $n\in\mathbb{Z}_{\geq 1}$ such that $G^k=\bigsqcup_{i=1}^n\textup{graph}\varphi_i$. Then we have $R_G=\bigoplus_{i=1}^\infty l^\infty_R(X)\cdot \alpha_{\varphi_i}$, where each $l^\infty_R(X)\cdot \alpha_{\varphi_i}$ is projective since we have 
\begin{equation*}
    l^\infty_R(X)\cdot \alpha_{\varphi_i}\simeq l^\infty_R(X)\cdot 1_{\textup{im}\varphi_i} ,\ f\alpha_{\varphi_i}\mapsto f1_{\textup{im}\varphi_i}.
\end{equation*}
\end{rmk}

For a left $R_G$-module $M$, the cohomology groups $\textup{H}^n(G,M)$ are defined in Definition \ref{intr6}. Note that a left $R_G$-module is also a left $\mathbb{Z}_G$-module since we have natural isomorphism $R_G\simeq \mathbb{Z}_G\otimes_\mathbb{Z}R$ of rings. The following lemma states that $\textup{H}^n(G,M)$ can be computed by regarding $M$ as a left $\mathbb{Z}_G$-module. We record this as a basic property, although it is not directly related to our goal.

\begin{lem}\label{bk11}
For all left $R_G$-module $M$ and integer $n\geq 0$, we have
\begin{equation*}
    \textup{Ext}_{\mathbb{Z}_G}^n(l^\infty_\mathbb{Z}(X),M)\simeq \textup{Ext}_{R_G}^n(l^\infty_R(X),M).
\end{equation*}
\end{lem}

\begin{proof}
Take a $\mathbb{Z}_G$-projective resolution
\begin{equation*}
    \cdots\to P_{n}\to \cdots \to P_0\to l^\infty_\mathbb{Z}(X)\to 0.
\end{equation*}
By Remark \ref{bk12}, this is also an $l^\infty_\mathbb{Z}(X)$-projective resolution. Then by Lemma \ref{pre3}, the sequence
\begin{equation}
    \cdots\to l^\infty_R(X)\otimes_{l^\infty_\mathbb{Z}(X)} P_{n}\to \cdots \to l^\infty_R(X)\otimes_{l^\infty_\mathbb{Z}(X)}P_0\to l^\infty_R(X)\to 0 \label{eqbk4}
\end{equation}
is exact. Note that we have $R_G\simeq l^\infty_R(X)\otimes_{l^\infty_\mathbb{Z}(X)}\mathbb{Z}_G$ as right $\mathbb{Z}_G$-modules. Indeed, the isomorphism is given by the map 
\begin{equation*}
    f\alpha_\varphi\in R_G\mapsto f\otimes \alpha_\varphi\in l^\infty_R(X)\otimes_{l^\infty_\mathbb{Z}(X)}\mathbb{Z}_G
\end{equation*}
for $\varphi\in\llangle G\rrangle$ and $f\in l^\infty_R(X)$. Then the sequence
\begin{equation*}
    \cdots\to R_G\otimes_{\mathbb{Z}_G} P_{n}\to \cdots \to R_G\otimes_{\mathbb{Z}_G} P_0\to l^\infty_R(X)\to 0
\end{equation*}
is an $R_G$-projective resolution since it is identified with the exact sequence \eqref{eqbk4}. Now for any left $R_G$-module $M$, we have an isomorphism of cochain complexes
\begin{equation*}
    \textup{Hom}_{\mathbb{Z}_G}(P_*,M)\simeq \textup{Hom}_{{R}_G}(R_G\otimes_{\mathbb{Z}_G}P_*,M),
\end{equation*}
given by
\begin{equation*}
    \textup{Hom}_{\mathbb{Z}_G}(P_n,M)\to \textup{Hom}_{{R}_G}(R_G\otimes_{\mathbb{Z}_G}P_n,M),\ \sigma\mapsto(1_X\otimes p\mapsto \sigma(p)).
\end{equation*}
Hence we have $\textup{Ext}_{\mathbb{Z}_G}^n(l^\infty_\mathbb{Z}(X),M)\simeq \textup{Ext}_{R_G}^n(l^\infty_R(X),M)$ for every $n\geq 0$.
\end{proof}

\begin{lem}\label{lot3}
Take $\{\varphi_i\}_{i=1}^n\subset\llangle G\rrangle$ so that $G=\bigsqcup_{i=1}^n\mleft(\textup{graph}\varphi_i\sqcup\textup{graph}\varphi_i^{-1}\mright)$. Then there exists an $R_G$-projective resolution
\begin{equation*}
    \cdots\to \bigoplus_{i=1}^n R_G \xrightarrow{\partial_1} R_G\xrightarrow{\varepsilon} l^\infty_R(X)\to 0
\end{equation*}
with $\varepsilon(1_X)=1_X$ and $\partial_1(1_X[i])=\alpha_{\varphi_i}-1_{\textup{im}\varphi_i}$. Here $1_X[i]\in \bigoplus_{i=1}^nR_G$ is the element such that the $i$-th coordinate is $1_X$ and the other coordinates are $0$.
\end{lem}

\begin{proof}
Since $\varepsilon$ is clearly surjective, it suffices to show that $\textup{ker}\varepsilon=\textup{im}\partial_1$. Note that $\varepsilon\circ\partial_1=0$ holds.

We claim that if $\psi$ is a composition of elements of $\left\{\varphi_i^{\pm1}\right\}_{i=1}^n$, then we have $1_{\textup{im}\psi}-\alpha_{\psi}\in\textup{im}\partial_1$. Indeed, it is clear that 
\begin{equation*}
    \alpha_{\varphi_i}-1_{\textup{im}\varphi_i},\ \alpha_{\varphi_i^{-1}}-1_{\textup{im}\varphi_i^{-1}}=\alpha_{\varphi_i^{-1}}(1_{\textup{im}\varphi_i}-\alpha_{\varphi_i})\in \textup{im}\partial_1.
\end{equation*}
If $\psi=\psi_1\psi_2$ and $\alpha_{\psi_1}-1_{\textup{im}\psi_1},\ \alpha_{\psi_2}-1_{\textup{im}\psi_2}\in\textup{im}\partial_1$, then we have
\begin{align*}
    \alpha_{\psi}-\alpha_{\psi_1}1_{\textup{im}\psi_2}=\alpha_{\psi_1}(\alpha_{\psi_2}-1_{\textup{im}\psi_2})\in\textup{im}\partial_1\\
    \alpha_{\psi_1}1_{\textup{im}\psi_2}-1_{\textup{im}\psi}=1_{\textup{im}\psi}(\alpha_{\psi_1}-1_{\textup{im}\psi_1})\in\textup{im}\partial_1
\end{align*}
and thus 
\begin{equation*}
    \alpha_{\psi}-1_{\textup{im}\psi}=(\alpha_{\psi}-\alpha_{\psi_1}1_{\textup{im}\psi_2})+(\alpha_{\psi_1}1_{\textup{im}\psi_2}-1_{\textup{im}\psi})\in\textup{im}\partial_1.
\end{equation*}
Hence the claim is proved.

Now let $a=\sum_{j=1}^mf_j\alpha_{\psi_j}\in \textup{ker}\varepsilon$,
where $f_j\in l^\infty_R(X)$ and $\psi_j$ are compositions of elements of $\left\{\varphi_i^{\pm1}\right\}_{i=1}^n$. Then we have
\begin{equation*}
    a=a-\varepsilon(a)=\sum_{j=1}^mf_j(\alpha_{\psi_j}-1_{\textup{im}\psi_j})\in \textup{im}\partial_1.
\end{equation*}
Hence $\textup{ker}\varepsilon=\textup{im}\partial_1$ holds.
\end{proof}

\begin{cor}\label{bk15}
Take $\{\varphi_i\}_{i=1}^n\subset\llangle G\rrangle$ so that $G=\bigsqcup_{i=1}^n\mleft(\textup{graph}\varphi_i\sqcup\textup{graph}\varphi_i^{-1}\mright)$. Then for every left $R_G$-module $M$, the cohomology group $\textup{H}^0(G,M)$ is naturally identified with the set
\begin{equation}
    \{\xi\in M\mid (\alpha_{\varphi_i}-1_{\textup{im}\varphi_i})\xi=0\  \forall i=1,...,n\}. \label{eqbk1}
\end{equation}
\end{cor}

\begin{proof}
Consider the $R_G$-projective resolution of $l^\infty_R(X)$ as in Lemma \ref{lot3}. The cohomology group $\textup{H}^0(G,M)$ is the kernel of 
\begin{equation*}
    \partial_1^*=-\circ\partial_1:\textup{Hom}_{R_G}(R_G,M)\to\textup{Hom}_{R_G}\mleft(\bigoplus_{i=1}^n R_G,M\mright).
\end{equation*}
We identify $\textup{Hom}_{R_G}(R_G,M)$ with $M$ by the map $\sigma\in \textup{Hom}_{R_G}(R_G,M)\mapsto \sigma(1_X)\in M$. Then $\xi\in M$ is in $\textup{ker}\partial_1^*$ if and only if $(\alpha_{\varphi_i}-1_{\textup{im}\varphi_i})\xi=0$ for every $i$. Hence $\textup{H}^0(G,M)$ is identified with the set \eqref{eqbk1}.
\end{proof}

\subsection{Cohomological dimension}\label{sbk2}

Recall that the $R$-cohomological dimension of $G$ is denoted by $\textup{cd}_R(G)$ (Definiton \ref{intr13}). Note that $\textup{cd}_R(G)=\textup{pd}_{R_G}(l^\infty_R(X))$ holds by Proposition \ref{pre2}, and $\textup{cd}_R(G)\leq \textup{cd}_\mathbb{Z}(G)$ holds by Lemma \ref{bk11}. First we compute the cohomological dimension of some examples.

\begin{exm}
As in Example \ref{intr9}, suppose that a group $\Gamma$ with a finite generating set $S\subset \Gamma\setminus\{e\}$ acts on a set $X$ freely, and set
\begin{equation*}
    G=\left\{(x,s^{\pm 1}x)\in X\times X\:\middle|\: x\in X,\ s\in S\right\}.
\end{equation*}
Recall that the algebra $R_G$ is isomorphic to the crossed product $R\Gamma\ltimes l^\infty_R (X)$. 

Now suppose that $\textup{cd}_R(\Gamma)\leq n$. This implies that there exists an $R\Gamma$-projective resolution
\begin{equation*}
    0\to P_{n}\to \cdots \to P_0\to R\to 0
\end{equation*}
of the trivial left $R\Gamma$-module $R$. Then the sequence
\begin{equation}
    0\to l^\infty_R(X)\otimes_RP_{n} \to \cdots \to l^\infty_R(X)\otimes_RP_0\to l^\infty_R(X)\to 0 \label{eqbk5}
\end{equation}
is exact by Lemma \ref{pre3}. Note that we have $R_G\simeq l^\infty_R(X)\otimes_RR\Gamma$ as right $R\Gamma$-modules. Then the sequence
\begin{equation*}
    0\to R_G\otimes_{R\Gamma} P_{n} \to \cdots \to R_G\otimes_{R\Gamma}P_0\to l^\infty_R(X)\to 0,
\end{equation*}
is an $R_G$-projective resolution since it is identified with the exact sequence \eqref{eqbk5}. Hence $\textup{cd}_R(G)\leq n$ holds.
\end{exm}

\begin{lem}\label{bk4}
If $G$ is acyclic, then we have $\textup{cd}_R(G)\leq 1$.
\end{lem}

\begin{proof}
Take $\{\varphi_i\}_{i=1}^n\subset\llangle G\rrangle$ so that $G=\bigsqcup_{i=1}^n\mleft(\textup{graph}\varphi_i\sqcup\textup{graph}\varphi_i^{-1}\mright)$. We will show that the sequence
\begin{equation*}
    0\to \bigoplus_{i=1}^n R_G\cdot 1_{\textup{im}\varphi_i}\xrightarrow{\partial_1} R_G\xrightarrow{\varepsilon} l^\infty_R(X)\to 0
\end{equation*}
given by $\varepsilon(1_X)=1_X$ and $\partial_1(1_{\textup{im}\varphi_i}[i])=\alpha_{\varphi_i}-1_{\textup{im}\varphi_i}$ is an $R_G$-projective resolution. By Lemma \ref{lot3}, it suffices to show that $\partial_1$ is injective.

Let $\sum_{i=1}^na_i[i]\in \bigoplus_{i=1}^n R_G\cdot 1_{\textup{im}\varphi_i}$. Fix $x\in X$. The set
\begin{equation}
    \bigsqcup_{i=1}^n\left\{(y,\varphi_i^{-1}y)\mid y\in [x]_G\cap \textup{im}\varphi_i\right\} \label{eqbk6}
\end{equation}
is identified with the set of $1$-simplices of the tree $([x]_G,G|_{[x]_G})$. Let $b_x$ be the finitely supported function on set \eqref{eqbk6} is defined by
\begin{equation*}
    (y,\varphi_i^{-1}(y))\mapsto a_i(x,y)
\end{equation*}
for $1\leq i\leq n$, and let $c_x$ be the finitely supported function on $[x]_G$ defined by
\begin{equation*}
    y\in [x]_G\mapsto\partial_1\mleft(\sum_{i=1}^na_i[i]\mright)(x,y)=\sum_{i=1}^n \mleft( a_i(x,\varphi_iy)1_{\textup{dom}\varphi_i}(y)-a_i(x,y)1_{\textup{im}\varphi_i}(y)\mright).
\end{equation*}
Then the map $b_x\mapsto c_x$ is exactly the boundary operator of the simplicial chain complex. Since $G|_{[x]_G}$ is a tree, this is injective for every $x$. Hence $\partial_1$ is also injective.
\end{proof}

Next we investigate some properties related to the condition $\textup{cd}_R(G)\leq 1$.

\begin{lem}\label{bk6}
If $\textup{cd}_R(G)\leq 1$ and $\textup{H}^1(G,R_G)=0$, then we have $\textup{cd}_R(G)=0$.
\end{lem}

\begin{proof}
This follows from Lemma \ref{bk7}.
\end{proof}

\begin{lem}\label{bk5}
We have $\textup{cd}_R(G)=0$ if and only if $E_G$ is uniformly finite, i.e., $\sup_{x\in X}|[x]_G|<\infty$.
\end{lem}

\begin{proof}
Suppose that $E_G$ is uniformly finite. Take $Y\subset X$ so that $|Y\cap [x]_G|=1$ for every $x\in X$. Now $l^\infty_R(X)$ is isomorphic to $R_G\cdot 1_Y$ as left $R_G$-modules by the map
\begin{equation*}
    f\in l^\infty_R(X)\mapsto ((x,y)\in E_G\mapsto f(x)1_Y(y)).
\end{equation*}
Hence the left $R_G$-module $l^\infty_R(X)$ is projective (see Remark \ref{bk17}), which means that $\textup{cd}_R(G)=0$.

To show the converse, suppose that $\textup{cd}_R(G)=0$. Then since the left $R_G$-module $l^\infty_R(X)$ is projective, we can take a semi-inverse of the map $\varepsilon:R_G\to l^\infty_R(X),\ 1_X\mapsto 1_X$, that is, there exists $\sigma\in \textup{Hom}_{R_G}(l^\infty_R(X),R_G)$ such that $\varepsilon\circ \sigma=\textup{id}_{l^\infty_R(X)}$. Fix $x\in X$. Then for all $y,z\in [x]_G$, we have
\begin{align*}
    \sigma(1_X)(x,z)=(1_{\{(y,x)\}}\sigma(1_X))(y,z)=\sigma(1_{\{(y,x)\}*}1_X)(y,z)\\
    =\sigma(1_{\{y\}})(y,z)=(1_{\{(y,y)\}}\sigma(1_X))(y,z)=\sigma(1_X)(y,z).
\end{align*}
Moreover, there exists $z\in [x]_G$ such that $\sigma(1_X)(x,z)\neq 0$ since $\varepsilon\circ \sigma(1_X)=1_X$ implies that $\sum_{z\in [x]_G}\sigma(1_X)(x,z)=1$. Then we have $\sigma(1_X)(y,z)\neq 0$ for all $y\in [x]_G$. Since $\sigma(1_X)\in R_G$, this implies that $\sup_{x\in X}\textup{diam}_G([x]_G)<\infty$. Hence $E_G$ is uniformly finite. 
\end{proof}

\begin{lem}\label{bk8}
If $\textup{cd}_R(G)\leq 1$, then $\textup{H}^1(G,R_G)$ is finitely generated as a right $R_G$-module.
\end{lem}

\begin{proof}
Take an $R_G$-projective resolution
\begin{equation*}
    \cdots\to \bigoplus_{i=1}^n R_G \xrightarrow{\partial_1} R_G\xrightarrow{\varepsilon} l^\infty_R(X)\to 0
\end{equation*}
as in Lemma \ref{lot3}. By $\textup{cd}_R(G)\leq 1$, the left $R_G$-module $\textup{im}\partial_1$ is projective by Proposition \ref{pre2} (iv), and thus we have another $R_G$-projective resolution
\begin{equation*}
    0\to \textup{im}\partial_1 \hookrightarrow R_G\xrightarrow{\varepsilon} l^\infty_R(X)\to 0.
\end{equation*}
By the computation of $\textup{H}^1(G,R_G)$ using this projective resolution, it suffices to show that $\textup{Hom}_{R_G}(\textup{im}\partial_1,R_G)$ is finitely generated as a right $R_G$-module. Since $\textup{im}\partial_1$ is projective, there exists a homomorphism $\sigma:\textup{im}\partial_1\to \bigoplus_{i=1}^n R_G$ such that $\partial_1\circ \sigma=\textup{id}_{\textup{im}\partial_1}$. Then there exists a surjective right $R_G$-homomorphism 
\begin{equation*}
    -\circ \sigma:\textup{Hom}_{R_G}\mleft(\bigoplus_{i=1}^n R_G,R_G\mright)\to \textup{Hom}_{R_G}(\textup{im}\partial_1,R_G).
\end{equation*}
Note that $\textup{Hom}_{R_G}(\bigoplus_{i=1}^n R_G,R_G)\simeq \bigoplus_{i=1}^n R_G$ holds as right $R_G$-modules, and thus it is finitely generated. Hence the right $R_G$-module $\textup{Hom}_{R_G}(\textup{im}\partial_1,R_G)$ is also finitely generated. The proof is done.
\end{proof}

The following two lemmas will be used in section \ref{smos2} to show that the condition $\textup{cd}_R(G)\leq 1$ is preserved under some operations.

\begin{lem}\label{lot2}
If $H$ is a coarsely embedded subgraph of $G$, then $\textup{cd}_R(H)\leq \textup{cd}_R(G)$ holds. In particular, if $G=H\ast K$ with $H$ and $K$ subgraphs of $G$, then we have $\textup{cd}_R(H)\leq \textup{cd}_R(G)$.
\end{lem}

\begin{proof}
Note that $R_H$ is a subalgebra of $R_G$ by Remark \ref{pre1}. It suffices to show that $R_G$ is projective as a left $R_H$-module. Indeed, if this is true, then any $R_G$-projective resolution is also an $R_H$-projective resolution, and the inequality $\textup{cd}_R(H)\leq \textup{cd}_R(G)$ follows from Proposition \ref{pre2} (iii). 

Take $\{\varphi_i\}_{i=0}^\infty\subset\llangle G\rrangle$ so that for every $k\in\mathbb{Z}_{\geq 0}$, there exists $n\in\mathbb{Z}_{\geq 0}$ such that $G^k=\bigcup_{i=1}^n\textup{graph}\varphi_i$. Then we define $\{\psi_i\}_{i=0}^\infty$ as follows: Set $\psi_0=\varphi_0$. For $i\geq 1$, let $\psi_i$ be the restriction of $\varphi_i$ such that
\begin{equation*}
    \textup{graph}\psi_i=\textup{graph}\varphi_i\setminus\bigcup_{j=0}^{i-1}\{(x,y)\in E_G\mid y\in\textup{dom}\varphi_j,\ (x,\varphi_jy)\in E_H\}.
\end{equation*}
Then the subsets
\begin{equation*}
    F_i=\{(x,y)\in E_G\mid y\in\textup{dom}\psi_i,\ (x,\psi_iy)\in E_H\}
\end{equation*}
satisfy
\begin{equation*}
    \bigcup_{j=0}^{i}\{(x,y)\in E_G\mid y\in\textup{dom}\varphi_j,\ (x,\varphi_jy)\in E_H\}=\bigsqcup_{j=0}^i F_i
\end{equation*}
for every $i\geq 0$ by construction, and thus we have $E_G=\bigsqcup_{i=0}^\infty F_i$. Now we show that $R_G=\bigoplus_{i=0}^\infty R_H\cdot \alpha_{\psi_i}$. Since all functions in $R_H\cdot \alpha_{\psi_i}$ are supported on $F_i$ and $\{F_i\}_i$ are disjoint, the $R_H$-sumbodules $\{R_H\cdot \alpha_{\psi_i}\}_{i=0}^\infty$ form a direct sum in $R_G$. Then it suffices to show that $\{\alpha_{\varphi_i}\}_{i=0}^\infty\subset \bigoplus_{i=0}^\infty R_H\cdot \alpha_{\psi_i}$. For $i\geq 0$, we have $\alpha_{\varphi_i}=\sum_{j=0}^i\alpha_{\varphi_i}|_{F_j}$ since $\textup{graph}\varphi_i\subset \bigsqcup_{j=0}^i F_i$. Note that if $a\in R_G$ is supported on $F_j$, then $a\alpha_{\psi_j^{-1}}$ is supported on $E_H$, and thus $a\alpha_{\psi_j^{-1}}\in R_H$ holds by Remark \ref{pre1}. Hence we have
\begin{equation*}
    \alpha_{\varphi_i}=\sum_{j=0}^i\mleft(\mleft(\alpha_{\varphi_i}|_{F_j}\mright)\alpha_{\psi_j^{-1}}\mright)\alpha_{\psi_j} \in \bigoplus_{j=0}^i R_H\cdot \alpha_{\psi_j}
\end{equation*}
for every $i\geq 0$, and thus $R_G=\bigoplus_{i=0}^\infty R_H\cdot \alpha_{\psi_i}$ as required. This implies that $R_G$ is projective as a left $R_H$-module.
\end{proof}

\begin{lem}\label{bk9}
If a graph $(Y,G')$ with uniformly bounded degrees is quasi-isometric to $(X,G)$, then $\textup{cd}_R(G)= \textup{cd}_R(G')$ holds.
\end{lem}

\begin{proof}
We can take $X_0\subset X$ so that $\gamma|_{X_0}:X_0\to Y$ is injective and $k=\sup_{x\in X}d_G(x,X_0)<\infty$. Then the graph 
\begin{equation*}
    G_0=\{(x_1,x_2)\in X_0\times X_0\mid 0<d_G(x_1,x_2)\leq 2k+1\}
\end{equation*}
generates the equivalence relation $E_G|_{X_0}$, and the inclusion map $(X_0,G_0)\to (X,G)$ is a quasi-isometry. Then $\gamma|_{X_0}:(X_0,G_0)\to (Y,G')$ is also a quasi-isometry. Hence it suffices to show that injective quasi-isometries preserve the cohomological dimension, and thus we may assume that $\gamma:X\to Y$ is injective.

The set $1_{\gamma(X)}\cdot R_{G'}$ is a left $R_G$-module in the following way: for $a\in R_G$ and $b\in 1_{\gamma(X)}\cdot R_{G'}$, the function $ab\in 1_{\gamma(X)}\cdot R_{G'}$ is defined by
\begin{equation*}
    (ab)(\gamma(x),y)=\sum_{z\in [x]_{G}}a(x,z)b(\gamma(z),y).
\end{equation*}
We claim that this left $R_G$-module is projective. Indeed, take $\{\varphi_i\}_{i=1}^n\subset\llangle G'\rrangle$ so that $\bigsqcup_{i=1}^n\textup{dom}\varphi_i=Y$ and $\bigcup_{i=1}^n\textup{im}\varphi_i\subset \gamma(X)$. Then we have 
\begin{equation*}
    1_{\gamma(X)}\cdot R_{G'}=\bigoplus_{i=1}^n1_{\gamma(X)}\cdot R_{G'}\cdot1_{\textup{dom}\varphi_i}\simeq \bigoplus_{i=1}^n R_G\cdot1_{\gamma^{-1}(\textup{im}\varphi_i)}
\end{equation*}
as left $R_G$-modules by the map
\begin{equation*}
    a\in R_G\cdot1_{\gamma^{-1}(\textup{im}\varphi_i)}\mapsto a\circ (\gamma^{-1}\times\gamma^{-1})\cdot\alpha_{\varphi_i}\in 1_{\gamma(X)}\cdot R_{G'}\cdot1_{\textup{dom}\varphi_i}.
\end{equation*}
Now if
\begin{equation*}
    0\to P_n\to \cdots \to P_0\to l^\infty_R(Y)\to 0
\end{equation*}
is an $R_{G'}$-projective resolution, then
\begin{equation*}
    0\to 1_{\gamma(X)}\cdot P_n\to \cdots \to 1_{\gamma(X)}\cdot P_0\to l^\infty_R(\gamma(X))\to 0
\end{equation*}
is an $R_G$-projective resolution of $l^\infty_R(\gamma(X))\simeq l^\infty_R(X)$. This implies that $\textup{cd}_R(G)\leq \textup{cd}_R(G')$.

Note that $R_{G'}\cdot 1_{\gamma(X)}$ is a projective right $R_G$-module in the same way as above. Now if 
\begin{equation*}
    0\to P_n\to \cdots \to P_0\to l^\infty_R(X)\to 0
\end{equation*}
is an $R_G$-projective resolution, then
\begin{align*}
    0\to \mleft(R_{G'}\cdot 1_{\gamma(X)}\mright)\otimes_{R_G}P_n\to \cdots \to \mleft(R_{G'}\cdot 1_{\gamma(X)}\mright)\otimes_{R_G}P_0\\
    \to \mleft(R_{G'}\cdot 1_{\gamma(X)}\mright)\otimes_{R_G}l^\infty_R(X)\to 0
\end{align*}
is an $R_{G'}$-projective resolution of $\mleft(R_{G'}\cdot 1_{\gamma(X)}\mright)\otimes_{R_G}l^\infty_R(X)\simeq l^\infty_R(Y)$. Hence we have $\textup{cd}_R(G')\leq \textup{cd}_R(G)$.
\end{proof}

\subsection{Structure of \texorpdfstring{\rm$\textup{H}^1(G,R_G)$}{H1(G,RG)}} \label{sbk3}

The goal of this subsection is to determine the cohomology group $\textup{H}^1(G,R_G)$.

We define the $R_G$-bimodule $W_G$ as follows:
\begin{equation*}
    W_G=\left\{\xi:E_G\rightarrow R\:\middle|\: \forall k\geq 0,\ \xi|_{G^k}\in l^\infty_R\mleft(G^k\mright)\right\},
\end{equation*}
and for $a,b\in R_G$ and $\xi\in W_G$,
\begin{equation*}
    (a\xi b)(x,y)=\sum_{z,w\in[x]_G}a(x,z)\xi(z,w)b(w,y).
\end{equation*}
Now we endow the set $\textup{Hom}_{l^\infty_R(X)}(R_G,l^\infty_R(X))$ with the structure of a left $R_G$-module such that: for $\sigma\in \textup{Hom}_{l^\infty_R(X)}(R_G,l^\infty_R(X))$ and $a,b\in R_G$,
\begin{equation*}
    (a\sigma)(b)=\sigma(ba).
\end{equation*}

\begin{rmk}\label{bk1}
For $\sigma\in\textup{Hom}_{l^\infty_R(X)}(R_G,l^\infty_R(X))$, $a\in R_G$ and $x\in X$, we have 
\begin{equation}
    \sigma(a)(x)=\sum_{y\in[x]_G}a(x,y) \sigma(1_{\{(x,y)\}})(x).\label{eqbk2}
\end{equation}
Indeed, since $1_{\{x\}}a=\sum_{y\in[x]_G}a(x,y)1_{\{(x,y)\}}$ holds, we have
\begin{align*}
    \sigma(a)(x)
    &=(1_{\{x\}}\sigma(a))(x)=\sigma(1_{\{x\}}a)(x)\\
    &=\sigma\mleft(\sum_{y\in[x]_G}a(x,y)1_{\{(x,y)\}}\mright)(x)\\
    &=\sum_{y\in[x]_G}a(x,y) \sigma(1_{\{(x,y)\}})(x).
\end{align*}
\end{rmk}

\begin{lem}
We have $W_G\simeq\textup{Hom}_{l^\infty_R(X)}(R_G,l^\infty_R(X))$ as left $R_G$-modules.
\end{lem}

\begin{proof}
Define the map $\Phi:W_G\rightarrow \textup{Hom}_{l^\infty_R(X)}(R_G,l^\infty_R(X)),\ \xi\mapsto \Phi_\xi$ by
\begin{equation*}
    \Phi_\xi(a)(x)=\sum_{y\in[x]_G}a(x,y)\xi(y,x)
\end{equation*}
for $\xi\in W_G$, $a\in R_G$ and $x\in X$. This is a homomorphism of left $R_G$-modules since for all $\xi\in W_G$, $a,b\in R_G$ and $x\in X$, we have
\begin{align*}
    \Phi_{a\xi}(b)(x)&=\sum_{y\in[x]_G} b(x,y)\cdot(a\xi)(y,x)\\
    &=\sum_{y,z\in[x]_G} b(x,y)a(y,z)\xi(z,x)\\
    &=\sum_{z\in[x]_G} (ba)(x,z)\cdot \xi(z,x)\\
    &=\Phi_\xi(ba)(x)\\
    &=(a\Phi_\xi(b))(x).
\end{align*}
\par
The inverse $\Psi:\textup{Hom}_{l^\infty_R(X)}(R_G,l^\infty_R(X)) \to W_G,\ \sigma\mapsto \Psi_\sigma$ of $\Phi$ is given by
\begin{equation*}
    \Psi_\sigma (y,x)=\sigma(1_{\{(x,y)\}})(x)
\end{equation*}
for $\sigma\in \textup{Hom}_{l^\infty_R(X)}(R_G,l^\infty_R(X))$ and $(x,y)\in E_G$. Indeed,
\begin{align*}
    \Psi_{\Phi_\xi}(y,x)=\Phi_\xi(1_{\{(x,y)\}})(x)
    =\sum_{z\in[x]_G}1_{\{(x,y)\}}(x,z)\xi(z,x)
    =\xi(y,x),
\end{align*}
and
\begin{align*}
    \Phi_{\Psi_\sigma}(a)(x)=\sum_{y\in[x]_G}a(x,y)\Psi_\sigma(y,x)
    =\sum_{y\in[x]_G}a(x,y) \sigma(1_{\{(x,y)\}})(x)
    =\sigma(a)(x),
\end{align*}
where the last equation follows from equation \eqref{eqbk2}.
\end{proof}

\begin{lem}\label{bk2}
We have $\textup{H}^n(G,W_G)=0$ for every $n\geq 1$.
\end{lem}

\begin{proof}
Take an $R_G$-projective resolution
\begin{equation}
    \cdots\to P_{n}\to P_{n-1} \to \cdots \to P_0\to l^\infty_R(X)\to 0. \label{eqbk3}
\end{equation}
Then we have
\begin{equation*}
    \textup{H}^n(G, W_G)
    =\textup{H}^n(\textup{Hom}_{R_G}(P_*,W_G)).
\end{equation*}
By the previous lemma, we have
\begin{align*}
    \textup{Hom}_{R_G}(P_i,W_G)
    &=\textup{Hom}_{R_G}\mleft(P_i,\textup{Hom}_{l^\infty_R(X)}(R_G,l^\infty_R(X))\mright)\\
    &=\textup{Hom}_{l^\infty_R(X)}(R_G\otimes_{R_G}P_i,l^\infty_R(X))\\
    &=\textup{Hom}_{l^\infty_R(X)}(P_i,l^\infty_R(X)).
\end{align*}
Here, the second identification is given by 
\begin{align*}
    &\sigma\in\textup{Hom}_{R_G}\mleft(P_i,\textup{Hom}_{l^\infty_R(X)}(R_G, l^\infty_R(X))\mright)\\
    &\mapsto (a\otimes p\mapsto\sigma(p)(a)) \in\textup{Hom}_{l^\infty_R(X)}(R_G\otimes_{R_G}P_i,l^\infty_R(X)).
\end{align*}
Since the sequence \eqref{eqbk3} is also an $l^\infty_R(X)$-projective resolution (see Remark \ref{bk12}), we have 
\begin{equation*}
    \textup{H}^n(G,W_G)=\textup{H}^n\mleft(\textup{Hom}_{l^\infty_R(X)}(P_*,l^\infty_R(X))\mright)=\textup{Ext}_{l^\infty_R(X)}^n(l^\infty_R(X),l^\infty_R(X))=0
\end{equation*}
for every $n\geq 1$.
\end{proof}

Now we define the right $R_G$-module $Z_G$ as follows:
For $k\in\mathbb{Z}_{\geq 0}$, set
\begin{equation*}
    Z_G^k=\left\{ \xi\in l^\infty_R(E_G)\middle|
    \begin{array}{l}
        \forall x\in X\;\forall y,z\in [x]_G\setminus B_G(k;x),\\
        (y,z)\in G\Rightarrow \xi(y,x)=\xi(z,x).
    \end{array}
    \right\},
\end{equation*}
In other words, the set $Z_G^k$ consists of $\xi\in l^\infty_R(E_G)$ such that the function $\xi(\cdot,x)$ is constant on each $G$-connected component of $[x]_G\setminus B_G(k;x)$. Set $Z_G=\bigcup_{k=0}^\infty Z_G^k$. For $\xi\in Z_G$ and $a\in R_G$, define
\begin{equation*}
    (\xi a)(y,x)=\sum_{z\in [x]_G} \xi(y,z)a(z,x).
\end{equation*}
This function belongs to $Z_G$. Indeed, if $\xi\in Z_G^k$ and $\varphi\in\llangle G\rrangle$ with $\textup{graph}\varphi\subset G^l$, then we have $(\xi\alpha_\varphi)(y,x)=\xi(y,\varphi x)1_{\textup{dom}\varphi}(x)$, and thus $\xi\alpha_\varphi\in Z_G^{k+l}$.

Note that $R_G$ is a submodule of $Z_G$. An embedding $\iota:l^\infty_R(X)\rightarrow Z_G$ is defined as follows: for $f\in l^\infty_R(X)$ and $(y,x)\in E_G$,
\begin{equation*}
    \iota(f)(y,x)=f(x).
\end{equation*}
Then the image $\iota(l^\infty_R(X))$ is a right $R_G$-submodule of $Z_G$.

\begin{lem}\label{bk3}
We have natural isomorphisms $p:\textup{H}^0(G,W_G)\rightarrow \iota(l^\infty_R(X))$ and $q: \textup{H}^0(G,W_G/R_G)\rightarrow Z_G/R_G$ of right $R_G$-modules. Moreover, the diagram
\begin{equation*}
    \begin{CD}
     \textup{H}^0(G,W_G) @>{\pi}>> \textup{H}^0(G,W_G/R_G) \\
  @V{p}VV                  @V{q}VV \\
     \iota(l^\infty_R(X))  @>{\rho}>> Z_G/R_G
  \end{CD}
\end{equation*}
is commutative, where $\pi$ is the map induced by the quotient $W_G\rightarrow W_G/R_G$ and $\rho$ is the composition of the inclusion $\iota(l^\infty_R(X))\rightarrow Z_G$ and the quotient $Z_G\rightarrow Z_G/R_G$.
\end{lem}

\begin{proof}
Take $\{\varphi_i\}_{i=1}^n\subset\llangle G\rrangle$ so that $G=\bigsqcup_{i=1}^n\mleft(\textup{graph}\varphi_i\sqcup \textup{graph}\varphi_i^{-1}\mright)$. Then by Corollary \ref{bk15}, $\textup{H}^0(G,W_G)$ is identified with the right $R_G$-submodule
\begin{align*}
    &\{\xi\in W_G\mid (\alpha_{\varphi_i}-1_{\textup{im}\varphi_i}) \xi=0\ \forall i=1,...,n\}\\
    =&\{\xi\in W_G\mid \forall x\in X\:\forall y,z\in [x]_G,\  \xi(y,x)=\xi(z,x)\}\\
    =&\iota(l^\infty_R(X)).
\end{align*}
Here the first equation follows from 
\begin{equation*}
    ((\alpha_{\varphi_i}-1_{\textup{im}\varphi_i}) \xi) (y,x)=
    1_{\textup{im}\varphi_i}(y)(\xi(\varphi_i^{-1}y,x)-\xi(y,x)).
\end{equation*}
This identification is the desired isomorphism $p$. Also, $\textup{H}^0(G,W_G/R_G)$ is identified with the right $R_G$-submodule
\begin{equation*}
    \{[\xi]\in W_G/R_G\mid (\alpha_{\varphi_i}-1_{\textup{im}\varphi_i}) [\xi]=0\ \forall i=1,...,n\}.
\end{equation*}
For $\xi\in W_G$, we have $(\alpha_{\varphi_i}-1_{\textup{im}\varphi_i}) \xi \in R_G$ if and only if there exists $k\geq 0$ such that the function
\begin{equation*}
    y\in [x]_G\mapsto 1_{\textup{im}\varphi_i}(y)(\xi(\varphi_i^{-1}y,x)-\xi(y,x))
\end{equation*}
is supported on $B_G(k;x)$ for all $x\in X$. It follows that $(\alpha_{\varphi_i}-1_{\textup{im}\varphi_i}) \xi \in R_G$ for all $1\leq i\leq n$ if and only if there exists $k\geq 0$ such that
\begin{equation*}
    \forall y,z\in [x]_G\setminus B_G(k;x),\ ((y,z)\in G\Rightarrow \xi(y,x)=\xi(z,x))
\end{equation*}
for all $x\in X$. Hence, the inclusion map $Z_G/R_G\rightarrow W_G/R_G$ is onto $\textup{H}^0(G,W_G/R_G)$. The inverse of this map is the desired isomorphism $q$. Then the map $\pi$ coincides with $\xi\in \iota(l^\infty_R(X))\mapsto [\xi]\in Z_G/R_G$. The lemma is proved.
\end{proof}

\begin{prop}\label{bk10}
The right $R_G$-module $\textup{H}^1(G,R_G)$ is isomorphic to
\begin{equation*}
    Z_G/(R_G+\iota(l^\infty_R(X))).
\end{equation*}
In particular, $\textup{H}^1(G,R_G)$ is finitely generated as a right $R_G$-module if and only if so is $Z_G$.
\end{prop}

\begin{proof}
The short exact sequence $0\rightarrow R_G\rightarrow W_G\rightarrow W_G/R_G\rightarrow 0$ induces the long exact sequence in cohomology:
\begin{align*}
    0&\rightarrow \textup{H}^0(G,R_G)\rightarrow \textup{H}^0(G,W_G)\xrightarrow{\pi} \textup{H}^0(G,W_G/R_G)\\
    &\rightarrow \textup{H}^1(G,R_G)\rightarrow \textup{H}^1(G,W_G)\rightarrow \textup{H}^1(G,W_G/R_G)\\
    &\rightarrow \cdots.
\end{align*}
Since $\textup{H}^1(G,W_G)=0$ holds by Lemma \ref{bk2}, we have
\begin{equation*}
    \textup{H}^1(G,R_G)\simeq \textup{H}^0(G,W_G/R_G)/\textup{im}\pi.
\end{equation*}
Then Lemma \ref{bk3} implies that it is isomorphic to $Z_G/(R_G+\iota(l^\infty_R(X)))$.

The latter assertion in the proposition follows from that $R_G$ is generated by $\{1_X\}$ and $\iota(l^\infty_R(X))$ is generated by $\{\iota(1_X)\}$ as right $R_G$-modules.
\end{proof}

\begin{cor}\label{bk13}
The right $R_G$-module $\textup{H}^1(G,R_G)$ is finitely generated if and only if the right $R_G$-module $Z_G$ is generated by $Z_G^k$ for some $k\in\mathbb{Z}_{\geq 0}$.
\end{cor}

\begin{proof}
It suffices to show that for every $k\in\mathbb{Z}_{\geq 0}$, the set $Z_G^k$ is contained in a finitely generated submodule of $Z_G$.
Fix $k\in\mathbb{Z}_{\geq 0}$. For $x\in X$, let $\{C_i(x)\}_{i=1}^n$ be a family of $G$-connected components of $[x]_G\setminus B_G(k;x)$ such that $[x]_G\setminus B_G(k;x)=\bigsqcup_{i=1}^nC_i(x)$. By allowing $C_i(x)=\varnothing$, we may assume that the number $n$ does not depend on $x$. Now set $\xi_i(\cdot,x)=1_{C_i(x)}$ for $x\in X$ and $1\leq i\leq n$. Then for every $\xi\in Z_G^k$, there exists $\{f_i\}_{i=1}^n\subset l^\infty_R(X)$ such that $\xi-\sum_{i=1}^n\xi_if_i\in R_G^k$. Hence $Z_G^k$ is contained in the submodule of $Z_G$ genetated by $\{\xi_i\}_{i=1}^n\cup \{1_X\}$.
\end{proof}

\subsection{Cutsets and \texorpdfstring{\rm$\textup{H}^1(G,R_G)$}{H1(G,RG)}} \label{sbk4}

Now we discuss properties of $\textup{H}^1(G,R_G)$ in terms of cutsets of $G$. 

\begin{dfn}
Let $\mathcal{C}$ be a cutset of $G$. A function $\xi\in Z_G$ is \textit{represented} by $\mathcal{C}$ if there exist $\{\xi_i\}_{i=1}^n\subset Z_G$ satisfying $\xi_i(\cdot,x)=1_{C_i(x)}$ with $C_i(x)\in\mathcal{C}\cup\{\varnothing\}$, and $\{f_i\}_{i=1}^n\subset l^\infty_R(X)$ such that $\xi-\sum_{i=1}^n\xi_if_i\in \iota(l^\infty_R(X))+R_G$.

We say that $\textup{H}^1(G,R_G)$ is \textit{generated} by $\mathcal{C}$ if every element of $Z_G$ is represented by $\mathcal{C}$.
\end{dfn}

Recall the definition of uniformly at most one-ended graphs (Definition \ref{intr11}).

\begin{prop}\label{intr10}
The following assertion holds:
\begin{enumerate}
    \item The right $R_G$-module $\textup{H}^1(G,R_G)$ is finitely generated if and only if it is generated by a cutset of $G$ with uniformly bounded boundaries.
    \item We have $\textup{H}^1(G,R_G)=0$ if and only if $G$ is uniformly at most one-ended.
\end{enumerate}
\end{prop}

\begin{proof}
(i) Suppose that $\textup{H}^1(G,R_G)$ is finitely generated as a right $R_G$-module. Then by Corollary \ref{bk13}, $Z_G$ is generated by $Z_G^k$ for some $k$. Let $[x]_G\setminus B_G(k;x)=\bigsqcup_{i=1}^n C_i(x)$ as in the proof of Corollary \ref{bk13}. Then every element of $Z_G^k$ is represented by the cutset
\begin{equation*}
    \mathcal{C}=\{C_i(x)\mid 1\leq i\leq n,\ x\in X,\ C_i(x)\neq\varnothing\},
\end{equation*}
which has uniformly bounded boundaries.
Note that if $\xi\in Z_G$ satisfies $\xi(\cdot,x)=1_{C(x)}$, then for $\varphi\in\llangle G\rrangle$, we have $(\xi\alpha_\varphi)(\cdot,x)=1_{C(\varphi x)}1_{\textup{dom}\varphi}(x)$. Hence every element of $Z_G$ is represented by the cutset $\mathcal{C}$.

Conversely, suppose that $\textup{H}^1(G,R_G)$ is generated by a cutset $\mathcal{C}$ of $G$ with 
\begin{equation*}
    \sup_{C\in\mathcal{C}}\textup{diam}_G\mleft(\partial_\textup{ov}^GC\mright)=k<\infty.
\end{equation*}
Let $\xi\in Z_G$ be such that $\xi(\cdot,x)=1_{C(x)}$ with $C(x)\in \mathcal{C}\cup\{\varnothing\}$. It suffices to show that $\xi$ is contained in the submodule of $Z_G$ generated by $Z_G^k$. Take $\varphi\in\llangle G\rrangle$ so that $\textup{im}\varphi=\{x\in X\mid C(x)\neq \varnothing\}$ and $\varphi^{-1}x\in\partial_\textup{ov}^GC(x)$ for every $x\in \textup{im}\varphi$. Then we have $\xi\alpha_{\varphi}\in Z_G^k$ since $(\xi\alpha_{\varphi})(\cdot,x)=1_{C(\varphi x)}$ and $\partial_\textup{ov}^GC(\varphi x)\subset B_G(k;x)$ for $x\in \textup{dom}\varphi$. Now $\xi=\xi\alpha_{\varphi}\alpha_{\varphi^{-1}}$ holds, and thus the claim is proved.

(ii) Suppose that $\textup{H}^1(G,R_G)=0$. For $k\in\mathbb{Z}_{\geq 0}$, let $[x]_G\setminus B_G(k;x)=\bigsqcup_{i=1}^n C_i(x)$ as in the proof of Corollary \ref{bk13}. Define $\xi_i\in Z_G$ by $\xi_i(\cdot,x)=1_{C_i(x)}$. Then we have $\xi_i\in \iota(l^\infty_R(X))+R_G$ for every $i$ by assumption. This implies that there exists $r>0$ such that for all $x\in X$ and $1\leq i\leq n$, either $\textup{diam}_G(C_i(x))\leq r$ or $\textup{diam}_G\mleft(\overline{C_i(x)}\mright)\leq r$ holds. Hence $G$ is uniformly at most one-ended.

Conversely, suppose that $G$ is uniformly at most one-ended. Let $\xi\in Z_G$ be defined by $\xi(\cdot,x)=1_{C(x)}$ with $C(x)\in \textup{Cut}(G)$. Since there exists $r>0$ such that for all $x\in X$, either $\textup{diam}_G(C(x))\leq r$ or $\textup{diam}_G\mleft(\overline{C(x)}\mright)\leq r$ holds, we have $\xi\in\iota(l^\infty_R(X))+R_G$. Hence we have $\textup{H}^1(G,R_G)=0$.
\end{proof}

The next lemma will be needed to show that the above properties are preserved after decomposing a graph. Recall Definitions \ref{intr3} and \ref{intr14}, and Notation \ref{pre5}.

\begin{lem}\label{mos6}
Suppose that $\textup{H}^1(G,R_G)$ is generated by a cutset $\mathcal{C}$ of $G$.
\begin{enumerate}
    \item If $(Y,G')$ is a graph with uniformly bounded degrees and $\lambda:(Y,G')\rightarrow (X,G)$ is a quasi-isometry, then $\textup{H}^1(G',R_{G'})$ is generated by the cutset ${\lambda^{-1}(\mathcal{C})}^*$ of $G'$, where
    \begin{equation*}
        {\lambda^{-1}(\mathcal{C})}=\{\lambda^{-1}(C)\mid C\in\mathcal{C}\}.
    \end{equation*}
    \item If $G=H\ast K$ with $H$ and $K$ subgraphs of $G$, then $\textup{H}^1(H,R_H)$ is generated by the cutset ${\mathcal{C}|_{H}}^*$ of $H$, where
    \begin{equation*}
        {\mathcal{C}|_{H}}=\{C\cap \omega\mid C\in\mathcal{C},\ \omega\in X/E_H\}.
    \end{equation*}
\end{enumerate}
\end{lem}

\begin{proof}
(i)  Take subsets $Y_k\subset Y\ (k=1,...,m)$ so that $Y=\bigsqcup_{k=1}^m Y_k$ and $\lambda|_{Y_k}$ is injective, and take a quasi-isometric inverse $\gamma:(X,G)\rightarrow (Y,G')$ of $\lambda$. Let $\eta\in Z_{G'}$ be defined by $\eta(\cdot,y)=1_{D(y)}$ with $D(y)\in\textup{Cut}(G')\cup\{\varnothing\}$. Since $\eta=\sum_{k=1}^m\eta 1_{Y_k}$, it suffices to show that $\eta 1_{Y_k}$ is represented by ${\lambda^{-1}(\mathcal{C})}^*$ for every $k$. Hence we fix $1\leq k\leq m$, and we may assume $\eta$ is supported on $Y\times Y_k$. We define $\xi\in Z_G$ supported on $X\times \lambda(Y_k)$ by
\begin{equation*}
    \xi(\cdot,\lambda(y))=1_{\gamma^{-1}(D(y))}
\end{equation*}
for $y\in Y_k$.
Since $\textup{H}^1(G,R_G)$ is represented by $\mathcal{C}$, there exist $\{\xi_i\}_{i=1}^n\subset Z_G$ satisfying $\xi_i(\cdot, x)=1_{C_i(x)}$ with $C_i(x)\in\mathcal{C}\cup\{\varnothing\}$, and $\{f_i\}_{i=1}^n\in l^\infty_R(X)$ such that
\begin{equation*}
    \xi-\sum_{i=1}^n \xi_if_i\in \iota(l^\infty_R(X))+R_G.
\end{equation*}
Then by composing $\lambda\times \lambda$, we have
\begin{equation*}
    \xi\circ (\lambda\times \lambda)-\sum_{i=1}^n (\xi_i\circ (\lambda\times \lambda)) (f_i\circ \lambda) \in \iota(l^\infty_R(Y))+R_{G'}.
\end{equation*}
Now for every $y\in Y_k$, we have $\xi(\lambda(\cdot),\lambda(y))=1_{(\gamma\circ \lambda)^{-1}(D(y))}$ by definition. Since $\gamma\circ \lambda$ is uniformly close to $\textup{id}_Y$, there exists $r>0$ such that
\begin{equation*}
    D(y)\bigtriangleup (\gamma\circ\lambda)^{-1}(D(y))\subset B_{G'}(r;y)
\end{equation*}
for every $y\in Y_k$. Hence we have $\eta- \xi\circ (\lambda\times\lambda)\in R_{G'}$ and thus
\begin{equation*}
    \eta - \sum_{i=1}^n (\xi_i\circ (\lambda\times \lambda)) (f_i\circ \lambda) \in \iota(l^\infty_R(Y))+R_{G'}.
\end{equation*}
Since $\xi_i(\lambda(\cdot),\lambda(y))=1_{\lambda^{-1}(C_i(\lambda(y)))}$ and $\lambda^{-1}(C_i(\lambda(y)))\in\lambda^{-1}(\mathcal{C})$ hold for $y\in Y_k$, the function $\eta$ is represented by ${\lambda^{-1}(\mathcal{C})}^*$.

(ii) For a cut $D\subset [x]_H$ of $H$, let $\widehat{D}\subset [x]_G$ be the cut of $G$ such that $\partial_{\textup{oe}}^G\widehat{D}=\partial_\textup{oe}^HD$. Note that such $\widehat{D}$ exists and satisfies $\widehat{D}\cap [x]_H=D$ since we have $G=H\ast K$.

Let $\eta\in Z_H$ be defined by $\eta(\cdot, x)=1_{D(x)}$ with $D(x)\in\textup{Cut}(H)\cup\{\varnothing\}$. We define $\widehat{\eta}\in Z_G$ by $\widehat{\eta}(\cdot,x)=1_{\widehat{D(x)}}$, where $\widehat{\varnothing}=\varnothing$. Since $\textup{H}^1(G,R_G)$ is generated by $\mathcal{C}$, there exist $\{\xi_i\}_{i=1}^n\subset Z_G$ satisfying $\xi_i(\cdot,x)=1_{C_i(x)}$ with $C_i(x)\in\mathcal{C}\cup\{\varnothing\}$, and $\{f_i\}_{i=1}^n\in l^\infty_R(X)$ such that
\begin{equation*}
    \widehat{\eta}-\sum_{i=1}^n \xi_if_i\in \iota(l^\infty_R(X))+R_G.
\end{equation*}
Then by restricting to $E_H$, we have 
\begin{equation*}
    \eta-\sum_{i=1}^n \xi_i|_{E_H}f_i\in \iota(l^\infty_R(X))|_{E_H}+R_H,
\end{equation*}
where $\iota(l^\infty_R(X))|_{E_H}=\{\iota(f)|_{E_H}\mid f\in l^\infty_R(X)\}$.
Since $\xi_i|_{E_H}(\cdot,x)=1_{C_i(x)\cap [x]_H}$ and $C_i(x)\cap [x]_H\in \mathcal{C}|_H$ hold, the function $\eta$ is represented by ${\mathcal{C}|_H}^*$.
\end{proof}

Now we show the converse of Theorem \ref{intr1}:

\begin{prop}\label{bk14}
Suppose that there exists a quasi-isometry $\gamma:(X,G)\to (Y,G')$, where $(Y,G')$ is a graph with uniformly bounded degrees such that:
\begin{enumerate}
    \item $G'=T\ast H$ with $T$ and $H$ subgraphs of $G'$.
    \item $T$ is acyclic.
    \item $H$ is uniformly at most one-ended.
\end{enumerate}
Then $\textup{H}^1(G,R_G)$ is finitely generated as a right $R_G$-module.
\end{prop}

\begin{proof}
For $t\in T$, there exists a unique cut $C_t$ of $G'$ such that $\partial_\textup{oe}^{G'}C_t=\{t\}$ since $T$ is acyclic and $G'=T\ast H$. We will show that $\textup{H}^1(G',R_{G'})$ is generated by the cutset $\mathcal{C}=\{C_t\mid t\in T\}$. Let $\xi\in Z_{G'}$ be defined by $\xi(\cdot,y)=1_{C(y)}$ with $C(y)\in \textup{Cut}(G')\cup\{\varnothing\}$. Since $H$ is uniformly at most one-ended, there exist $r>0$ such that if $C(y)$ intersects an $E_H$-class $\omega$ non-trivially (i.e., $C(y)\cap \omega\neq \varnothing,\omega$), then either $\textup{diam}_H(\omega\cap C(y))\leq r$ or $\textup{diam}_H(\omega\setminus C(y))\leq r$ holds. Then we remove the subset $\omega\cap C(y)$ from $C(y)$ if the former holds, and add the subset $\omega\setminus C(y)$ to $C(y)$ if the latter holds. This operation changes the function $\xi$ only up to $R_{G'}$ since we only need to consider $E_H$-classes $\omega$ such that $H|_\omega\cap \partial_\textup{ie}^{G'}C(y)\neq \varnothing$. Hence we may assume that for every $y\in Y$, the set $C(y)$ is $E_H$-invariant, which implies $\partial_\textup{ie}^{G'}C(y)\subset T$. Moreover, since $C(y)$ has finitely many $G'$-connected components (and the number of them is uniformly bounded), we may assume that every $C(y)$ is $G'$-connected. Then we have $[y]_{G'}\setminus C(y)=\bigsqcup_{t\in\partial_\textup{ie}^{G'}C(y)}C_t$ since $T$ is acyclic and $G'=T\ast H$. Hence $\xi$ is represented by $\mathcal{C}$, and $\textup{H}^1(G',R_{G'})$ is generated by $\mathcal{C}$.

By Lemma \ref{mos6} (i), the set $\textup{H}^1(G,R_G)$ is generated by the cutset $\gamma^{-1}(\mathcal{C})^*$, which has uniformly bounded boundaries. This proves the proposition by Proposition \ref{intr10} (i).
\end{proof}

The following technical lemma will be used in section \ref{smos1}.

\begin{lem}\label{mos5}
Let $\mathcal{C}=\{C_i\}_{i\in I}$ and $\mathcal{C}'=\{C_i'\}_{i\in I}$ be cutsets of $G$ indexed by a set $I$. Suppose that there exists $r>0$ such that for every $C_i\in\mathcal{C}$, we have
\begin{equation*}
    C_i'\subset B_G(r;C_i)\textup{ and }\overline{C_i'}\subset B_G\mleft(r;\overline{C_i}\mright).
\end{equation*}
Then every $\xi\in Z_G$ represented by $\mathcal{C}$ is also represented by $\mathcal{C}'$.
\end{lem}

\begin{proof}
We claim that for every $i\in I$, we have
\begin{equation*}
    C_i\bigtriangleup C_i' \subset B_G\mleft(r;\partial_\textup{ov}^GC_i\mright).
\end{equation*}
Indeed, for $x\in C_i\setminus C_i'$, we have $x\in C_i\cap B_G\mleft(r;\overline{C_i}\mright)$ by assumption. Then any $G$-path from $x$ to a point in $\overline{C_i}$ of length at most $r$ must intersect $\partial_\textup{ov}^GC_i$, and thus $x\in B_G\mleft(r;\partial_\textup{ov}^GC_i\mright)$. We can show $C_i'\setminus C_i \subset B_G\mleft(r;\partial_\textup{ov}^GC_i\mright)$ in the same way.

Let $\xi\in Z_G$ be defined by $\xi(\cdot,x)=1_{C(x)}$ with $C(x)\in\mathcal{C}\cup\{\varnothing\}$. Take $k\in\mathbb{Z}_{\geq 0}$ so that $\xi\in Z_G^k$. For $x\in X$, let $C'(x)=C_i'$ if $C(x)=C_i$ and let $C'(x)=\varnothing$ if $C(x)=\varnothing$. Define $\xi'\in Z_G$ by $\xi'(\cdot,x)=1_{C'(x)}$. Then for every $x\in X$ with $C(x)\in \mathcal{C}$, we have
\begin{equation*}
    C(x)\bigtriangleup C'(x)\subset B_G\mleft(r;\partial_\textup{ov}^GC(x)\mright)\subset B_G(r+k;x),
\end{equation*}
which implies that $\xi-\xi'\in R_G$. Since $\xi'$ is represented by $\mathcal{C}'$, so is $\xi$.
\end{proof}

\section{Proof of the theorems}\label{smos}

\subsection{Proof of Theorem \ref{intr1}} \label{smos1}

Let $R$ be a non-zero commutative ring and $(X,G)$ a Borel graph with uniformly bounded degrees. The following lemma is a key for the proof of Theorem \ref{intr1}.

\begin{lem}\label{mos1}
Let $\mathcal{C}$ be a Borel treeset of $G$ with uniformly bounded boundaries. Then there exist a Borel graph $(Y,G')$ with uniformly bounded degrees and an injective Borel quasi-isometry $\gamma:(X,G)\rightarrow (Y,G')$ such that:
\begin{enumerate}
    \item $G'=T\ast H$, with $T$ and $H$ Borel subgraphs of $G'$.
    \item $T$ is acyclic.
    \item If $\lambda:(Y,G')\rightarrow (X,G)$ is a Borel quasi-isometric inverse of $\gamma$ with $\lambda\circ\gamma=\textup{id}_X$, then every element of $Z_H$ represented by $\lambda^{-1}(\mathcal{C}){|_H}^*$ is trivial in $\textup{H}^1(H,R_H)$.
\end{enumerate}
\end{lem}

\begin{proof}
First, we construct the standard Borel space $Y$ by adding new vertices on edges of $G$. Take a Borel subset $G^+\subset G$ so that for every $(x_0,x_1)\in G$, we have $|\{(x_0,x_1),(x_1,x_0)\}\cap G^+|=1$. By Lemma \ref{lot4}, we have
\begin{equation*}
    N:=\sup_{(x_0,x_1)\in G^+}|\{C\in\mathcal{C}\mid x_0\in C,\ x_1\notin C\}|<\infty.
\end{equation*}

Let $e=(x_0,x_1)\in G^+$. Since $\mathcal{C}$ is nested, the elements of the set 
\begin{equation*}
    \{C\in\mathcal{C}\mid x_0\in C,\ x_1\notin C\}
\end{equation*}
are ordered as $C_0^e\subsetneq C_1^e\subsetneq \cdots\subsetneq C_n^e$ with $n\leq N-1$. Then we add $n$ vertices $y_1^e,\ldots,y_n^e$ on the edge $e$ in order from the side of $x_0$. In other words, we replace the edge $e=(x_0,x_1)$ by the graph $\{(y_i^e,y_{i+1}^e)\}_{i=0}^n$, where $y_0^e=x_0$ and $y_{n+1}^e=x_1$. We do this procedure for all $e\in G^+$, and then the oriented graph $(X,G^+)$ is changed into a new oriented graph, which we denote by $(Y,G_1^+)$. Set $G_1=G_1^+\sqcup (G_1^+)^{-1}$. Then $(Y,G_1)$ is naturally a Borel graph since the set of added vertices 
\begin{equation*}
    Y\setminus X=\{y_i^e\mid 1\leq i\leq N \textup{ such that } C_i^e\textup{ exists}\}
\end{equation*}
is identified with the Borel set
\begin{equation*}
    \{(e,i)\in G^+\times\mathbb{Z} \mid 1\leq i\leq N \textup{ such that } C_i^e\textup{ exists}\}.
\end{equation*}
Note that the inclusion map $(X,G)\rightarrow (Y,G_1)$ is an injective Borel quasi-isometry.

Now for $C\in\mathcal{C}$, we define the cut $C'$ of $G_1$ as follows: For $x\in X$, let $x\in C'$ if and only if $x\in C$. For $y_i^e\in Y\setminus X$,
\begin{itemize}
    \item if $C$ does not separate the endpoints of $e$, then let $y_i^e\in C'$ if and only if the endpoints of $e$ are in $C$,
    \item if $C=C_j^e$, then let $y_i^e\in C'$ if and only if $i\leq j$, and
    \item if $\overline{C}=C_j^e$, then let $y_i^e\in C'$ if and only if $i> j$.
\end{itemize}
Then $\mathcal{C}'=\{C'\}_{C\in\mathcal{C}}$ is a Borel treeset of $G_1$ with uniformly bounded boundaries. Indeed, it is nested and closed under complementation since the map $C\in\mathcal{C}\mapsto C'\in\mathcal{C}'$ preserves the order and complements. Also, we have
\begin{equation*}
    \partial_\textup{iv}^{G_1}C'\subset B_{G_1}\mleft(N;\partial_\textup{iv}^GC\mright)
\end{equation*}
for every $C\in\mathcal{C}$, and thus $\mathcal{C}'$ has uniformly bounded boundaries. These imply that $\mathcal{C}'$ is a Borel treeset of $G_1$ by Lemma \ref{lot1}. 
Moreover for every $(x,y)\in G_1$, we have
\begin{equation*}
    |\{C'\in\mathcal{C}'\mid x\in C'\ y\notin C'\}|\leq 1
\end{equation*}
by construction.

Then by Proposition \ref{lot8}, there exists a Borel graph $G'$ on $Y$ Lipschitz equivalent to $G_1$ such that:
\begin{itemize}
    \item $G'=T\ast H$ with $T$ and $H$ Borel subgraphs of $G'$. 
    \item $T$ is acyclic.
    \item $\mathcal{C}'{|_H}^*=\varnothing$.
\end{itemize}

It is left to show condition (iii).
Let $\lambda:(Y,G')\rightarrow (X,G)$ be a Borel quasi-isometric inverse of the inclusion $(X,G)\rightarrow (Y,G')$ with $\lambda|_X=\textup{id}_X$.
Since $G_1$ and $G'$ are Lipschitz equivalent, there exists $r>0$ such that $C'\subset B_{G_1}(N;C)\subset B_{G'}(r; C)$ for every $C\in\mathcal{C}$.
Then for every $C\in\mathcal{C}$, we have
\begin{equation}
    C'\subset B_{G'}\mleft(r;\lambda^{-1}(C)\mright)\ \textup{ and }\ 
    \overline{C'}\subset B_{G'}\mleft(r;\overline{\lambda^{-1}(C)}\mright) \label{eqmos1}
\end{equation}
since $C\subset\lambda^{-1}(C)$ and $\mathcal{C}$ is closed under complementation.

Now let $\xi\in Z_H$ be defined by $\xi(\cdot,y)=1_{\lambda^{-1}(C(y))\cap [y]_H}$ with $C(y)\in\mathcal{C}\cup\{\varnothing\}$. Define $\widehat{\xi}\in Z_{G'}$ by $\widehat{\xi}(\cdot,y)=1_{\lambda^{-1}(C(y))}$. By Lemma \ref{mos5} and claim \eqref{eqmos1}, the function $\widehat{\xi}$ is represented by $\mathcal{C'}$. This implies that $\xi=\widehat{\xi}|_{E_H}$ is represented by ${\mathcal{C}'|_H}^*=\varnothing$. Hence $\xi$ is trivial in $\textup{H}^1(H,R_H)$.
\end{proof}

 We also need the following lemma.

\begin{lem}\label{lot5}
Suppose that $G=H\ast K$ with $H$ and $K$ Borel subgraphs of $G$. Let $\gamma:(X,H)\rightarrow (Y,H')$ be an injective Borel quasi-isometry. Then $E_{H'}$ and $E_{(\gamma\times\gamma)(K)}$ are freely intersecting, and $\gamma:(X,G)\rightarrow (Y,H'\ast (\gamma\times\gamma)(K))$ is also a Borel quasi-isometry.
\end{lem}

\begin{proof}
Let $\{y_i\}_{i=0}^{2n}\subset Y$ be such that
\begin{equation*}
    (y_{2i},y_{2i+1})\in E_{H'}\setminus\Delta_Y\textup{ and }(y_{2i+1},y_{2i+2})\in E_{(\gamma\times\gamma)(K)}\setminus\Delta_Y
\end{equation*}
for $0\leq i\leq n-1$. Suppose toward a contradiction that $y_0= y_{2n}$. Since $\gamma$ is injective, we can take $(x_{2i-1},x_{2i})\in E_K\setminus \Delta_X$ so that $(\gamma(x_{2i-1}),\gamma(x_{2i}))=(y_{2i-1},y_{2i})$ for $1\leq i\leq n$. We set $x_0=x_{2n}$. Then we have $(x_{2i},x_{2i+1})\in E_{H}\setminus\Delta_X$ for $0\leq i\leq n-1$ since $E_H=(\gamma\times\gamma)^{-1}(E_{H'})$, which follows from that $\gamma$ is a quasi-isometry. This contradicts $G=H\ast K$. Hence $E_{H'}$ and $E_{(\gamma\times\gamma)(K)}$ are freely intersecting.

By Remark \ref{pre4}, $\gamma:(X,H)\rightarrow (Y,H')$ is $l$-biLipschitz for some $l\geq 1$. Note that $d_K=d_{(\gamma\times\gamma)(K)}\circ(\gamma\times\gamma)$. Let $\{x_i\}_{i=0}^{2n}\subset X$ be such that
\begin{align*}
    (x_0,x_1)\in E_H,\ (x_{2i-1},x_{2i})\in E_K\setminus\Delta_X,\\
    (x_{2i},x_{2i+1})\in E_H\setminus\Delta_X\textup{ and }(x_{2n-1},x_{2n})\in E_K
\end{align*}
for every $1\leq i\leq n-1$.
Set $G'=H'\ast (\gamma\times\gamma)(K)$. Then we have
\begin{align*}
    d_G(x_0,x_{2n})&=\sum_{i=0}^{n-1}(d_H(x_{2i},x_{2i+1})+d_K(x_{2i+1},x_{2i+2}))\textup{ and}\\
    d_{G'}(\gamma(x_0),\gamma(x_{2n}))
    &=\sum_{i=0}^{n-1}(d_{H'}(\gamma(x_{2i}),\gamma(x_{2i+1}))+d_{(\gamma\times\gamma)(K)}(\gamma(x_{2i+1}),\gamma(x_{2i+2})).
\end{align*}
Thus
\begin{align*}
    d_{G'}(\gamma(x_0),\gamma(x_{2n}))
    &\leq \sum_{i=0}^{n-1}(ld_H(x_{2i},x_{2i+1})+d_K(x_{2i+1},x_{2i+2}))\leq ld_G(x_0,x_{2n}) \textup{ and }\\
    d_{G'}(\gamma(x_0),\gamma(x_{2n}))
    &\geq \sum_{i=0}^{n-1}(l^{-1}d_H(x_{2i},x_{2i+1})+d_K(x_{2i+1},x_{2i+2}))\geq
    l^{-1}d_G(x_0,x_{2n}).
\end{align*}
Hence we have $l^{-1}d_G|_{E_G}\leq d_{G'}\circ(\gamma\times\gamma)|_{E_G}\leq ld_G|_{E_G}$.

We can also verify that $d_{G'}\circ(\gamma\times\gamma)<\infty$ implies $d_G<\infty$ by the same argument as in the first paragraph of this proof.
Finally we have
\begin{equation*}
    \sup_{y\in Y} d_{G'}(y,\gamma(X))\leq \sup_{y\in Y} d_{H'}(y,\gamma(X))<\infty.
\end{equation*}
Hence $\gamma:(X,G)\rightarrow (Y,G')$ is a Borel quasi-isometry.
\end{proof}

Now we write Theorem \ref{intr1} again, and prove it.

\begin{thm}\label{mos3}
Let $R$ be a non-zero commutative ring and $(X,G)$ a Borel graph with uniformly bounded degrees. Suppose that $\textup{H}^1(G,R_G)$ is finitely generated as a right $R_G$-module. Then there exists an injective Borel quasi-isometry $\gamma:(X,G)\to (Y,G')$, where $(Y,G')$ is a Borel graph with uniformly bounded degrees such that:
\begin{enumerate}
    \item $G'=T\ast H$ with $T$ and $H$ Borel subgraphs of $G'$.
    \item $T$ is acyclic.
    \item $H$ is uniformly at most one-ended.
\end{enumerate}
\end{thm}

\begin{proof}
Since $\textup{H}^1(G,R_G)$ is a finitely generated right $R_G$-module, by Proposition \ref{intr10} (i), there exists $k\geq 0$ such that $\textup{H}^1(G,R_G)$ is generated by the cutset
\begin{equation*}
    \mathcal{C}=\{C\in \textup{Cut}(G)\mid \textup{diam}_G(\partial_\textup{iv}^GC\cup \partial_\textup{ov}^GC)\leq k\}.
\end{equation*}
Note that this is a Borel cutset closed under complementation.
By Lemma \ref{lot7}, there exist Borel treesets $\mathcal{C}_i\ (i=0,1,...,n)$ of $G$ such that $\mathcal{C}=\bigcup_{i=0}^n\mathcal{C}_i$.

Applying Lemma \ref{mos1} to the Borel treeset $\mathcal{C}_0$, there exist a Borel graph $(X_0,G_0)$ with uniformly bounded degrees and an injective Borel quasi-isometry $\gamma_0:(X,G)\rightarrow (X_0,G_0)$ such that:
\begin{enumerate}
\setcounter{enumi}{3}
    \item $G_0=T_0\ast H_0$ with $T_0$ and $H_0$ Borel subgraphs of $G_0$.
    \item $T_0$ is acyclic.
    \item If $\lambda_0:(X_0,G_0)\rightarrow (X,G)$ is a Borel quasi-isometric inverse of $\gamma_0$ with $\lambda_0\circ\gamma_0=\textup{id}_{X}$, then every element of $Z_{H_0}$ represented by $\lambda_0^{-1}(\mathcal{C}_0){|_{H_0}}^*$ is trivial in $\textup{H}^1(H_0,R_{H_0})$.
\end{enumerate}
Note that for $0\leq i\leq n$, the set $\lambda_0^{-1}(\mathcal{C}_i){|_{H_0}}^*$ is a Borel treeset of $H_0$. Indeed, by Lemma \ref{lot11}, it is a Borel cutset with uniformly bounded boundaries which is nested and closed under complementation. This implies that it is a Borel treeset by Lemma \ref{lot1}. Now by Lemma \ref{mos6} and condition (vi), the set $\textup{H}^1(H_0,R_{H_0})$ is generated by the cutset $\bigcup_{i=1}^n\mathcal{C}_i^0$, where $\mathcal{C}_i^0=\lambda_0^{-1}(\mathcal{C}_i){|_{H_0}}^*$.

We will inductively construct a sequence of injective Borel quasi-isometries $\gamma_k:(X_{k-1},G_{k-1})\rightarrow (X_k,G_k)$ such that each $(X_k,G_k)$ is a Borel graph with uniformly bounded degrees satisfying:
\begin{enumerate}
\setcounter{enumi}{6}
    \item $G_k=T_k\ast H_k$ with $T_k$ and $H_k$ Borel subgraphs of $G_k$.
    \item $T_k$ is acyclic.
    \item $\textup{H}^1(H_k,R_{H_k})$ is generated by a cutset $\bigcup_{i=k+1}^n\mathcal{C}_i^k$, where each $\mathcal{C}_i^k$ is a Borel treeset of $H_k$ with uniformly bounded boundaries.
\end{enumerate}

Suppose that $(X_{k-1},G_{k-1})$ is defined. Then applying Lemma \ref{mos1} to the Borel graph $(X_{k-1},H_{k-1})$ and the Borel treeset $\mathcal{C}_k^{k-1}$, there exist a Borel graph $(X_k,H_k')$ with uniformly bounded degrees and an injective Borel quasi-isometry $\gamma_k:(X_{k-1},H_{k-1})\to (X_k,H_k')$ such that:
\begin{enumerate}
\setcounter{enumi}{9}
    \item $H_k'=T_k'\ast H_k$ with $T_k'$ and $H_k$ Borel subgraphs of $H_k'$.
    \item $T_k'$ is acyclic.
    \item If $\lambda_k:(X_k,H_k')\rightarrow (X_{k-1},H_{k-1})$ is a Borel quasi-isometric inverse of $\gamma_k$ with $\lambda_k\circ\gamma_k=\textup{id}_{X_{k-1}}$, then every element of $Z_{H_k}$ represented by $\lambda_k^{-1}{{(\mathcal{C}_k^{k-1})}|_{H_k}}^*$ is trivial in $\textup{H}^1(H_k,R_{H_k})$.
\end{enumerate}
By assumption (ix) on $(X_{k-1},G_{k-1})$ and condition (xii), the set $\textup{H}^1(H_k,R_{H_k})$ is generated by the cutset $\bigcup_{i=k+1}^n\mathcal{C}_i^k$, where $\mathcal{C}_i^k=\lambda_k^{-1}{(\mathcal{C}_i^{k-1})|_{H_k}}^*$ is a Borel treeset of $H_k$. Now by Lemma \ref{lot5} (applying to $(X,G)=(X_{k-1},G_{k-1})$, $H=H_{k-1}$, $K=T_{k-1}$ and $\gamma=\gamma_k:(X_{k-1},H_{k-1})\to (X_k,H_k')$), the equivalence relations $E_{H_k'}$ and $E_{(\gamma_k\times \gamma_k)(T_{k-1})}$ are freely intersecting, and the map
\begin{equation*}
    \gamma_k:(X_{k-1},G_{k-1})\rightarrow (X_k, H_k'\ast(\gamma_k\times \gamma_k)(T_{k-1}))
\end{equation*}
is also a Borel quasi-isometry. Now we set
\begin{align*}
    T_k&= T_k'\ast (\gamma_k\times \gamma_k)(T_{k-1}) \textup{ and}\\
    G_k&=H_k'\ast(\gamma_k\times \gamma_k)(T_{k-1})=H_k\ast T_k.
\end{align*}
Then these satisfy conditions (vii)-(ix) as desired. Hence we obtain the sequence $\gamma_k:(X_{k-1},G_{k-1})\rightarrow (X_k,G_k)\ (k=1,...,n)$.

Now $\textup{H}^1(H_n,R_{H_n})$ is generated by the cutset $\bigcup_{i=n+1}^n\mathcal{C}_i^n$, which is empty. This implies that $\textup{H}^1(H_n,R_{H_n})=0$, and thus $H_n$ is uniformly at most one-ended by Proposition \ref{intr10} (ii). Then $(Y,G')=(X_n,G_n)$, $T=T_n$ and $H=H_n$ satisfy conditions (i)-(iii), and
\begin{equation*}
    \gamma=\gamma_n\circ\cdots\circ\gamma_1\circ\gamma_0:(X,G)\to (Y,G')
\end{equation*}
is an injective Borel quasi-isometry.
\end{proof}

\begin{rmk}\label{mos2}
We remark about a feature of the above proof compared to Dunwoody's proof of Theorem \ref{intr4} (see also \cite[IV, Theorem 7.5]{DD}). One of the differences is that they provide a concrete construction of a single treeset separating all ends of the group (which is invariant under the group action). A similar thing is possible in our setting. Indeed, by the argument in the proof of \cite[Theorem 4.3]{Jar}, for any locally finite Borel graph $(X,G)$ and any integer $n\geq 1$, there exists a Borel treeset of $G$ with edge boundaries of size at most $n$ which generates the Boolean algebra on $X$ generated by all such cuts. However since we are concerned with the metric (not the size) of the boundaries of cuts, this machinery is not fit for our purpose. Therefore, instead of taking a single nice treeset, we divide the cutset into finitely many treesets and decompose the Borel graph several times. This might be impossible for general locally finite Borel graphs, and it is one of the reasons why we restrict our attention to the case of uniformly bounded degrees.
\end{rmk}

\subsection{Proof of Theorem \ref{intr2}} \label{smos2}

We recall the statement of Theorem \ref{intr2}:

\begin{thm}\label{mos4}
Let $R$ be a non-zero commutative ring and $(X,G)$ a Borel graph with uniformly bounded degrees. Then the following conditions are equivalent:
\begin{enumerate}
    \item There exists a Borel acyclic graph on $X$ Lipschitz equivalent to $G$.
    \item $\textup{cd}_R(G)\leq 1$.
\end{enumerate}
\end{thm}

\begin{proof}
Implication (i)$\Rightarrow$(ii) follows from Lemma \ref{bk4}. To show the converse, we assume $\textup{cd}_R(G)\leq 1$. Then $\textup{H}^1(G,R_G)$ is finitely generated as a right $R_G$-module by Lemma \ref{bk8}. By Theorem \ref{mos3}, there exist a Borel graph $(Y,G')$ with uniformly bounded degrees and an injective Borel quasi-isometry $\gamma:(X,G)\to (Y,G')$ such that:
\begin{itemize}
    \item $G'=T_1\ast H$ with $T_1$ and $H$ Borel subgraphs of $G'$.
    \item $T_1$ is acyclic.
    \item $\textup{H}^1(H,R_H)=0$.
\end{itemize}
Since a Borel quasi-isometry preserves the cohomological dimension by Lemma \ref{bk9}, we have $\textup{cd}_R(G')\leq 1$, which implies $\textup{cd}_R(H)\leq 1$ by Lemma \ref{lot2}. Then the conditions $\textup{cd}_R(H)\leq 1$ and $\textup{H}^1(H,R_H)=0$ imply that $\textup{cd}_R(H)=0$ by Lemma \ref{bk6}, and thus $E_H$ is uniformly finite by Lemma \ref{bk5}. Hence $H$ can be replaced by an Borel acyclic graph $T_2$ on $Y$ Lipschitz equivalently. Then the Borel graph $T'=T_1\ast T_2$ is acycilc and Lipschitz equivalent to $G'$. Note that $\gamma:(X,G)\to (Y,T')$ is also a Borel quasi-isometry.

Now we define a Borel quasi-isometric inverse $\lambda:(Y,T')\to (X,G)$ of $\gamma$ as follows: Let $\leq_\textup{B}$ be a Borel linear order on $X$. For $y\in Y$, let $\lambda(y)$ be the $\leq_\textup{B}$-minimal element of the set
\begin{equation*}
    \{x\in X \mid d_{T'}(y,\gamma(x))=d_{T'}(y,\gamma(X))\}.
\end{equation*}
Then for $(y,y')\in E_{T'}$, if $d_{T'}(y,\gamma(\lambda(y)))=d_{T'}(y,y')+d_{T'}(y',\gamma(\lambda(y)))$, then we have $\lambda(y)=\lambda(y')$. This implies that $T'|_{\lambda^{-1}(x)}$ is a subtree of $T'$ for every $x\in X$. Now we set
\begin{equation*}
    T=\left\{(x,x')\in E_G \setminus\Delta_X\:\middle|\:\mleft(\lambda^{-1}(x)\times\lambda^{-1}(x')\mright) \cap T'\neq\varnothing\right\}.
\end{equation*}
Note that $\lambda:(Y,T')\to (X,T)$ is the map contracting every subtree $T'|_{\lambda^{-1}(x)}$ to the point $x$. In particular, $T$ is acyclic. Moreover $\lambda:(Y,T')\to (X,T)$ is a quasi-isometry since all $\lambda^{-1}(x)$ are uniformly bounded. Then
\begin{equation*}
    \textup{id}_X=\lambda\circ\gamma:(X,G)\to (Y,T')\to (X,T)
\end{equation*}
is also a quasi-isometry, which implies that $G$ and $T$ are Lipschitz equivalent.
\end{proof}

\section*{Acknowledgments} 
I would like to thank my advisor, Professor Yoshikata Kida, for his continuous mentorship throughout my master’s and Ph.D. studies. This thesis was completed thanks to his helpful comments and suggestions. I am also grateful to Professor Anush Tserunyan for her warm hospitality during my visit to McGill University and for many valuable discussions. This work was supported by JSPS KAKENHI Grant Number 23KJ0653, and the WINGS-FMSP program in the University of Tokyo.

\footnotesize
\textsc{Graduate School of Mathematical Sciences, the University of Tokyo, 3-8-1 Komaba, Tokyo 153-8914, Japan.}\par
E-mail address: \texttt{ishikura8000@g.ecc.u-tokyo.ac.jp}

\end{document}